\documentclass[12pt]{amsart}
\usepackage[english]{babel}
\usepackage{amsmath,amsthm,amsfonts,amssymb,epsfig,color,ulem,vector,hyperref}
\usepackage[left=1in,top=1in,right=1in]{geometry}
\usepackage{enumitem}

\newtheorem{thm}{Theorem}[section]
\newtheorem{lem}[thm]{Lemma}
\newtheorem{prop}[thm]{Proposition}
\newtheorem{cor}[thm]{Corollary}
\newtheorem{df}[thm]{Definition}
\newtheorem{rem}[thm]{Remark}
\newtheorem{conj}[thm]{Conjecture}
\newtheorem{claim}[thm]{Claim}
 
\newcommand{\E}{\mathbb{E}}

\newcommand{\PP}{\mathbb{P}}
\newcommand{\R}{\mathbb{R}}
\newcommand{\Z}{\mathbb{Z}}
\newcommand{\N}{\mathbb{N}}
\newcommand{\G}{\mathcal{G}}
\newcommand{\J}{J}
\newcommand{\var}{\text{\rm Var}}
\newcommand{\e}{\varepsilon}
\newcommand{\ee}{{\rm e}}
\newcommand{\rd}{{\rm d}}

\hypersetup{colorlinks=true, linkcolor=blue, urlcolor=blue, citecolor=blue}

\begin{document}

\title[Disorder Chaos and Incongruent States in Spin Glasses]{A Relation between Disorder Chaos and Incongruent States in Spin Glasses on $\Z^d$}

\author[L.-P. Arguin]{L.-P. Arguin}            
 \address{L.-P. Arguin\\ 
 Department of Mathematics\\
 City University of New York, Baruch College and Graduate Center\\
 New York, NY 10010}
\email{louis-pierre.arguin@baruch.cuny.edu}

\author[C.M. Newman]{C.M. Newman}            
 \address{C.M. Newman\\ 
 Courant Institute of Mathematical Sciences\\
 New York, NY 10012 USA\\
  and NYU-ECNU Institute of Mathematical Sciences at NYU Shanghai\\
  3663 Zhongshan Road North, Shanghai 200062, China}
\email{newman@cims.nyu.edu}

\author[D.L.~Stein]{D.L.~Stein}            
 \address{D.L.~Stein\\ 
 Department\ of Physics and Courant Institute of Mathematical Sciences\\
  New York University\\
 	 New York, NY 10003, USA\\
	 and NYU-ECNU Institutes of Physics and Mathematical Sciences at NYU Shanghai\\
	  3663 Zhongshan Road North\\
	  Shanghai, 200062, China}
\email{daniel.stein@nyu.edu}

\date{}


\maketitle

\begin{abstract}
We derive lower bounds for the variance of the difference of energies between incongruent ground states, i.e., states with edge overlaps strictly less than one, of the Edwards-Anderson model on $\Z^d$.
The bounds highlight a relation between the existence of incongruent ground states 
and the absence of edge disorder chaos.
In particular, it suggests that the presence of disorder chaos is necessary for the variance to be of order less than the volume.
In addition, a relation is established between the scale of disorder chaos and the size of critical droplets.
The results imply a long-conjectured relation between the droplet theory of Fisher and Huse and the absence of incongruence.
\end{abstract}
 

\section{Introduction}
The Edwards-Anderson (EA) model is a nearest-neighbor model of a realistic spin glass in finite dimensions~\cite{EA75}.
As opposed to the infinite-range version, the Sherrington-Kirkpatrick (SK) model~\cite{SK75}, 
the critical behavior of the EA~model and in particular the existence of a phase transition and the nature of this phase transition remain 
elusive from both mathematical and physical perspectives. 
We refer to~\cite{Mac84,BM85,MPV87,FH88,KM00,PY00,NS03} and references therein for more details on competing pictures for the low-temperature thermodynamic structure of the EA~model.
In the case of the SK model, it is known that there exist at low enough temperature states with edge
overlap\footnote{In the SK model, unlike in the EA~model, edge and spin overlaps are trivially related.} strictly less than one~\cite{Parisi79,Parisi83,MPSTV84a,MPSTV84b,MPV87}. Such states are said to be {\it incongruent}.
The question of existence of incongruent ground states at zero temperature for the EA model in finite dimensions is the main motivation of the present paper.
More concretely, we relate the existence of such incongruent states to non-trivial lower bounds for the variance of the difference of ground state energies, which we relate in turn
to the presence and extent of edge disorder chaos. 
\\

\subsection{Background}

Consider a finite subset $\Lambda\subset \Z^d$; $\Lambda$ is considered to be a cube centered at the origin with side-length $L$ so that $|\Lambda|=L^d$.
The set of nearest-neighbor edges $\{x,y\}$ with $|x-y|=1$ and $x,y\in \Lambda$ is denoted by $\Lambda^*$.
We denote the {\it couplings} on $(\Z^d)^{*}$, the set of all nearest-neighbor edges of $\Z^d$, by
$\J=(J_{xy}, \{x,y\}\in (\Z^d)^{*})$.
We suppose that the couplings are independent and identically distributed Gaussian random variable with mean~$0$ and variance $1$. 
The distribution of $J$ is denoted by $\nu$.

The EA Hamiltonian on $\Lambda\subset \Z^d$ for the disorder $J$ is the Ising-type Hamiltonian with random couplings $J$:
\begin{equation}
\label{eqn: H}
H_{\Lambda, J}(\eta)=\sum_{\{x,y\}\in \Lambda^*} -J_{xy}\eta_x\eta_y\ ,
\end{equation}
where $\eta\in \{-1,+1\}^\Lambda$ is a {\it spin configuration} in $\Lambda$.

\begin{df}
\label{df: GS}
A spin configuration $\sigma\in \{-1,+1\}^{\Z^d}$ is a {\it ground state} for the EA Hamiltonian for the couplings $J$ if for every finite subset $\mathcal B$ of $\Z^d$
the configuration $\sigma$ restricted to $\mathcal B$ minimizes
\begin{equation}
\label{eqn: GS}
H_{\mathcal B, J}(\eta)+ \sum_{\{x,y\}\in \partial \mathcal B } -J_{xy} \eta_x\sigma_y \ \text{over  $\eta\in\{-1,+1\}^\mathcal B$} ,
\end{equation}
\noindent where $\partial \mathcal B$ stands for the edges with one vertex $x$ in $\mathcal B$ and one vertex $y$ in $\mathcal B^c$.
\end{df}
The minimizer of \eqref{eqn: GS} is unique $\nu$-a.s.~for the boundary condition given by $\sigma$ in $\mathcal B^c$.
The above definition is equivalent to the property that for any finite subset $\mathcal B$ of $ \Z^d$
\begin{equation}
\label{eqn: GS2}
\sum_{\{x,y\}\in \partial \mathcal B} J_{xy}\sigma_x\sigma_y\geq 0\ .
\end{equation}
Consider the {\it edge overlap} between $\sigma^1$, $\sigma^2$ in $\Lambda$:
\begin{equation}
\label{overlap}
Q_{\Lambda}(\sigma^1,\sigma^2)=\frac{1}{|\Lambda^*|} \sum_{\{x,y\}\in \Lambda^*} \sigma^1_x\sigma^1_y \sigma^2_x\sigma^2_y \ .
\end{equation}
Two ground states are said to be {\it incongruent} if
\begin{equation}
\label{eqn: incongruent}
\limsup_{\Lambda \to \Z^d} Q_{\Lambda}(\sigma^1, \sigma^2)<1\ . 
\end{equation}
In other words, there is a strictly positive fraction of edges in $\Lambda$ for which $\sigma_x^1\sigma^1_y\neq \sigma_x^2\sigma^2_y$.\\

We write $\G(J)\subset \{-1,+1\}^{\Z^d}$ for the set of infinite-volume ground states for the couplings $J$. 
In Section \ref{sect: metastate}, we recall the construction of certain measures on $\G(J)$ from limits of finite-volume ground states with specified boundary conditions.
Such a measure will be denoted by $\kappa_J$ and referred to as a {\it metastate}.
From these measures, it is possible to study three questions: 

\medskip

\begin{enumerate}[label=(\roman*)]
\item Is there more than one sub-sequential limit $\kappa_J$ along an infinite sequence of volumes?
\item How many ground states are in the support of $\kappa_J$ ? 
\item Do there exist two or more {\it incongruent\/} ground states in the support of these measures?
\end{enumerate}

\medskip 

To study these questions, we consider the probability measure $\PP$ on triples $(J,\sigma^1, \sigma^2)$ where
\begin{equation}
\label{eqn: prob}
\rd \PP= \rd \nu (J) \times \rd \kappa^{(1)}_{J}(\sigma^1)\times \rd \kappa^{(2)}_{J}(\sigma^2)\ ,
\end{equation}
where $\kappa_J^{(1)}, \kappa_J^{(2)}$ are two metastates. 
The measure $\PP$ samples the disorder $J$ and then two ground states for that disorder according to $\kappa^{(1)}_J\times\kappa^{(2)}_J$.
Questions (i), (ii), and (iii) were answered for the half-plane in \cite{ADNS10} (see also \cite{AD14} for general results on the set of ground states).
This paper is mainly concerned with Question (iii) for the model on $\Z^d$.
Question (iii) is narrower than (i) and (ii) in general, except when periodic boundary conditions are considered.
In that particular case, $\PP$ is translation-invariant and the existence of a single edge where $\sigma^1$ and $\sigma^2$ differ ensures 
the existence of a positive density of such edges.

\subsection{Main Results}

It is possible to modify the couplings locally under the measure $\rd \PP$ as follows. 
First, we redefine $\PP$ to add an extra independent copy $J'$ of the couplings:
\begin{equation}
\label{eqn: prob3}
\rd \PP= \rd \nu (J)\times  \rd \nu (J')\times  \rd \kappa^{(1)}_{J}(\sigma^1)\times \rd \kappa^{(2)}_{J}(\sigma^2)\ .
\end{equation}
We consider an interpolation $J(t)$ parametrized by $t\geq 0$ where
\begin{equation}
\label{eq:t}
J_{xy}(t)=\ee^{-t}J_{xy}+\sqrt{1-\ee^{-2t}}J'_{xy} \qquad \text{ if $\{x,y\}\in \Lambda^*$,}
\end{equation}
and $J_{xy}(t)=J_{xy}$ if $\{x,y\}\notin \Lambda^*$.
For each ground state $\sigma^1$ and $\sigma^2$, we will construct in Section~\ref{sect: metastate} a measurable map $t\mapsto \sigma^i(t)$, $i=1,2$, which for each $t$ gives a ground state for the value of the interpolated couplings at $t$. 
(We slightly abuse notation here since we use $\sigma^i$ for the map as well as for the initial point $\sigma^i=\sigma^i(0)$.)
It turns out that the distribution of the ground states $\sigma^i(t)$ under $\kappa^{(i)}_{J}$ is exactly the one of $\sigma^i$ under $\kappa^{(i)}_{J(t)}$, cf.~Section \ref{sect: metastate}.

The first main result of this paper is to establish a lower bound for the variance of the difference of ground state energies in terms of local coupling modifications.
\begin{thm}
\label{thm: main}
For all $t>0$, 
\begin{equation}
\label{eqn: variance bound}
\begin{aligned}
&\var\Big(H_{\Lambda,\J}(\sigma^1)-H_{\Lambda,\J}(\sigma^2)\Big)\geq  \\
&2|\Lambda^*|\int_0^t\left\{\E\Big[1-Q_\Lambda(\sigma^1,\sigma^2)\Big]-\sum_{i=1,2} \Big(2\cdot\E\big[1-Q_\Lambda(\sigma^i,\sigma^i(s))\big]\Big)^{1/2}\right\}\ee^{-s}\rd s \ .
\end{aligned}
\end{equation}
\end{thm}
The main interest of this bound is the explicit connection between incongruence, represented by the first expectation, and disorder chaos, or rather the absence thereof, represented by the second expectation.
\begin{df}[Absence of Disorder Chaos]
\label{df: ADC}
We say that there is absence of disorder chaos at scale $\alpha$, $0\leq \alpha\leq 1$, for $\PP$, if for any $\e>0$
there exist $A_\e$ with $\PP(A_\e)>1-\e$ and $C=C(\e)>0$, such that
$$
Q_\Lambda(\sigma^i, \sigma^i(t))>1-\e \text{ on $A_\e$, $i=1,2$,}
$$
 for all $t\leq C|\Lambda|^{-\alpha}$ and all $\Lambda$ large enough.
\end{df}
In other words, there is absence of disorder chaos at scale $\alpha$ if with large probability, the fraction of edges for which $\sigma^1(t)$ is different from $\sigma^1(0)$
remains small for $t\leq C|\Lambda|^{-\alpha}$. 
Let $\mathcal I$ be the event that incongruent states exist, that is
\begin{equation}
\mathcal I=\left\{(J,\sigma^1, \sigma^2): \limsup_{\Lambda \to \Z^d} Q_{\Lambda}(\sigma^1, \sigma^2)<1\right\}\ .
\end{equation}
Definition~\ref{df: ADC} and Theorem~\ref{thm: main} imply:
\begin{cor}
\label{cor: ADC-> variance}
Let $\PP$ be as in Equation \eqref{eqn: prob} with $\PP(\mathcal I)>0$.  If there is absence of disorder chaos at scale $0\leq \alpha \leq 1$, then there exists $C>0$ independent of $\Lambda$ such that 
\begin{equation}
\label{eqn: var bound 1}
\var\Big(H_{\Lambda,\J}(\sigma^1)-H_{\Lambda,\J}(\sigma^2)\Big)\geq C |\Lambda|^{1-\alpha}\ .
\end{equation}
\end{cor}

\bigskip

The second main result is a relation between the size of critical droplets and the absence of disorder chaos as above.
Fix an edge $b=\{x_0,y_0\}$ in a box $\Lambda$. As a function of $J_b$, the ground state is locally constant, cf.~Section \ref{sect: metastate}.
Roughly speaking, according to \eqref{eqn: GS2}, the ground state changes as $J_b$ is increased (or decreased) when the energy of the boundary (which of course passes through~$b$) of some connected cluster of spins first becomes negative. 
This connected cluster could be infinite. We write $\mathcal D_b$ for the subset of vertices of this cluster inside $\Lambda$.
We refer to $\mathcal D_b$ as the {\it critical droplet of the edge $b$ in $\Lambda$}. 
The important quantity is the size of the boundary $\partial \mathcal D_b$ containing the edges at the boundary with one vertex in $\mathcal D_b$ and one in its complement; see Figure \ref{fig: droplet} below for an illustration.

The next theorem relates the size of critical droplet boundaries to the absence of disorder chaos.
\begin{thm}
\label{thm: ADC}
Let $\PP$ be as in Equation \eqref{eqn: prob}.
Suppose that there exist $0\leq \gamma\leq 1$ and $C<\infty$ (independent of $\Lambda)$ such that with probability one,
for all large $\Lambda$,
\begin{equation}
\label{eqn: droplet ass}
|\partial\mathcal D_b| \leq C |\Lambda|^\gamma \text{ for all $b\in |\Lambda^*|$.}
\end{equation}
Then there is absence of disorder chaos for $\PP$ at every scale $\alpha>2\gamma$.
\end{thm}
\begin{rem}
\rm
Assumption \eqref{eqn: droplet ass} is a statement about the distribution of the size of the droplet. Indeed, we have by a union bound that
$$
\PP(\exists b\in \Lambda^*: |\partial \mathcal D_b|>a )\leq \sum_{b\in \Lambda^*} \PP(|\partial \mathcal D_b|>a )\ .
$$
Therefore, taking $a=a(\Lambda)$,  the assumption \eqref{eqn: droplet ass} would be satisfied if the tail distribution decays fast enough to ensure summability. 
\end{rem}
Together with Corollary \ref{cor: ADC-> variance}, this shows that non-trivial bounds on the variance of the difference of the ground state energies can be obtained by estimating the size of the critical droplets. Theorem \ref{thm: ADC} is probably far from optimal as it only gives non-trivial variance bounds for $\gamma<1/2$. It is easy to check that $\gamma=0$ at $d=1$, and one might expect that $\gamma=0$ also in $d=2$. More precise estimates combining the geometry of the droplets and their energy are needed to improve the result -- see Remark \ref{rem: estimate}. As a modest first step in this direction, we get that the variance is uniformly bounded away from zero.
\begin{cor}
\label{cor: variance one}
Let $\PP$ be as in Equation \eqref{eqn: prob} with $\PP(\mathcal I)>0$. Then one has for some constant $C>0$ independent of $\Lambda$, 
\begin{equation}
\label{eqn: var bound 2}
\var\Big(H_{\Lambda,\J}(\sigma^1)-H_{\Lambda,\J}(\sigma^2)\Big)\geq C \ .
\end{equation}
\end{cor}

\subsection{Relations to Other Results}

A variance lower bound for the difference of ground state energies was proved in~\cite{ANSW16} under the assumption that the average (over the metastate) of the edge correlation function differs for $\sigma^1$ and $\sigma^2$. 
There the variance lower bound was obtained by an adaptation of the martingale approach of \cite{AW90}. 
The corresponding result at positive temperature was proved in \cite{ANSW14}. 
As in \cite{AW90}, the variance bound in \cite{ANSW14, ANSW16} is based on the elementary inequality
\begin{equation}
\label{eqn: var before}
\var\Big(H_{\Lambda,\J}(\sigma^1)-H_{\Lambda,\J}(\sigma^2)\Big)\geq \var\left(\E\big[H_{\Lambda,\J}(\sigma^1)-H_{\Lambda,\J}(\sigma^2)\ \big|J_\Lambda\big]\right)\ .
\end{equation}
A non-trivial variance lower bound can then be proved if there is an inherent asymmetry between $\sigma^1$ and $\sigma^2$ {\it on average}
over $\kappa^{(1)}_J\times\kappa^{(2)}_J$ and  over all couplings but the ones in $\Lambda$. 
For ferromagnetic models, such as random-field ferromagnets, this is not a problem as the plus and minus states retain such an asymmetry. 
In~\cite{ANSW14,ANSW16}, the needed asymmetry arose as a consequence of the assumption of the existence of incongruence in spin glasses. Of course, this assumption might not hold in general.

A novel approach used in the present paper is to obtain variance lower bounds by {\it conditioning on the disorder outside $\Lambda$}.
In effect, we use the asymmetry between incongruent states that always exists when the couplings $J_{\Lambda^c}$ outside $\Lambda$  are fixed, and the ground states $(\sigma^1, \sigma^2)$ (always assumed to be incongruent) for this choice of couplings outside $\Lambda$ are also fixed. One can then think of the ground states in $\Lambda$ for the {\it boundary condition} $\sigma^1$ as a function of $J_\Lambda$: $J_\Lambda\mapsto \sigma^1(J_\Lambda)$. 
By conditioning on $(J_{\Lambda^c},\sigma^1, \sigma^2)$ instead of $J_\Lambda$ as in \eqref{eqn: var before}, we get that the variance is bounded below by
\begin{equation}
\label{eqn: idea}
 \E\Big[\var\big(H_{\Lambda,\J}(\sigma^1)-H_{\Lambda,\J}(\sigma^2)\ \big |J_{\Lambda^c},\sigma^1, \sigma^2\big)\Big]\ .
 \end{equation}
It turns out that the couplings $J_\Lambda$ are independent of $(J_{\Lambda^c},\sigma^1, \sigma^2)$, and thus remain Gaussian, cf.~Lemma \ref{lem: metastate}.
Therefore, variance lower bounds can be obtained on $\var\big(H_{\Lambda,\J}(\sigma^1)-H_{\Lambda,\J}(\sigma^2)\ \big |J_{\Lambda^c},\sigma^1, \sigma^2\big)$ using Gaussian methods.\\

\medskip

When no magnetic field is present, disorder chaos in the mathematical literature often refers simply to the overlap $Q_\Lambda(\sigma^1, \sigma^1(t))$ being close to $0$ with large probability for some positive $t$; see e.g.~\cite{C14}. 
With this definition, absence of disorder chaos means that $Q_\Lambda(\sigma^1, \sigma^1(t))$ is bounded away from $0$ for $t$ small. 
For example, for the EA model, Chatterjee~\cite{C14} showed absence of disorder chaos in this sense by proving the bound \footnote{The bound is proved in finite volume for fixed boundary conditions, but also holds when the boundary conditions are sampled from a metastate.}
$$
\E[Q_{\Lambda}(\sigma^1, \sigma^1(t))]\geq C q\ee^{-t/(Cq)}
$$
for some constant $C>0$ and $q=1/(4d^2)$; see \cite{C14}. 
This bound is {\it a priori} too weak to get a good lower bound using Theorem \ref{thm: main}. 
This is because it does not preclude that $\sigma^1(t)$ has overlap strictly smaller than one with $\sigma^1(0)$, and thus severely differs from $\sigma^1(0)$ for very small $t$. 
Absence of disorder chaos (in the above sense) was also proved for some range of $t$ depending on the size of the system for $p$-spin spherical spin glasses by Subag in~\cite{Subag16}.

In the physics literature, where the concept arose, the definition of disorder chaos (or the closely related temperature chaos) is slightly more nuanced, in that both occur only beyond a lengthscale related to the size of the perturbation $t$~\cite{BM87,FH88,KB05}. Our Definition~\ref{df: ADC} is simply a formalized version of the standard physics definition, with an emphasis on the {\it scale\/} of disorder chaos represented by the parameter~$\alpha$.

The central result of this paper is the connection established through Theorem~\ref{thm: main} and Corollary~\ref{cor: ADC-> variance} between the scale of edge disorder chaos and the size of fluctuations in incongruent ground state energies, which in turn has a direct bearing on the possible presence or absence of incongruence in short-range spin glasses~\cite{ANSW14,ANSW16,Stein16}. A further, unanticipated relation is established in Theorem~\ref{thm: ADC}, in which the size of critical droplets is shown to set the scale of disorder chaos, creating a direct link between the stability of spin glass ground states (through the size of their critical droplets) and ground state multiplicity.

Finally, the results proved in this paper shed light on predictions made by the droplet theory of spin glasses~\cite{FH86,FH87,HF87,FH88}, based on scaling approaches~\cite{Mac84,BM85}, on the absence of incongruence in short-range spin glasses. This will be taken up in Sect.~\ref{sec:scaling}.

\subsection{Structure of the Paper}

The necessary background about measures on ground states is given in Section 2. 
In Section 3, we prove Theorem \ref{thm: main} based on standard Gaussian interpolation applied to the conditional variance \eqref{eqn: idea}.
The proofs of Theorem \ref{thm: ADC} and those of Corollaries \ref{cor: ADC-> variance} and Corollary \ref{cor: variance one} appear in Section 4. 
Finally, Section 5 discusses the connection between the approach developed here and the droplet theory of Fisher \& Huse.

\medskip

\noindent{\bf Acknowledgements.}
The authors thank the referee for many important suggestions that led to simplifications of some of the proofs.
The authors are also grateful to Nick Read for insightful remarks and for an important correction to Proposition 2.9, as well as Aernout van Enter and Jon Machta for useful comments on the manuscript. 
The research of LPA is supported in part by U.S.~NSF
Grant~DMS-1513441 and by U.S.~NSF CAREER~DMS-1653602.
The research of CMN was supported in part by U.S.~NSF Grants DMS-1207678 and DMS-1507019.
The research of DLS was supported in part by U.S.~NSF Grant DMS-1207678.

\section{Local Modification of Couplings}
\label{sect: metastate}

In this section, we develop the necessary framework to address the dependence of infinite-volume ground states on local modifications of couplings.
This theory of local excitations is based on a previous construction of the {\it excitation metastate}, see \cite{NS01,ADNS10, ANSW14, AD11}.
The main results are Propositions \ref{prop: criterion} and \ref{prop: flex diff} which together yield sufficient conditions, in terms of the size of the critical droplet of a given edge, for a ground state to remain the same at that edge under local modification of couplings.

Throughout this section and henceforth, we will sometimes use the notation $\sigma_e=\sigma_x\sigma_y$ for the spin interaction at the edge $e=\{x,y\}$.
We will also fix the finite box $\Lambda\subset \Z^d$, and sometimes omit its dependence in the notation.

\subsection{Measures on Ground States and Local Excitations}
We first construct a measure on the set of ground states $\mathcal G(J)$ in terms of finite-volume ones.
Consider a box $\mathcal B_n=[-n,n]^d$ on $\Z^d$ and the EA Hamiltonian on $\mathcal B_n$ with specified boundary condition $\xi$
\begin{equation}
\label{eqn: H_n}
H_{\mathcal B_n, J}(\eta)= \sum_{e\in \mathcal B_n^*}-J_e\eta_e + \sum_{\{x,y\}\in \partial  \mathcal B_n} -J_{xy} \eta_x \xi_y\ .
\end{equation}
The ground state for this Hamiltonian is the unique $\nu$-a.s.~minimizer over all $\eta\in \{-1,+1\}^{\mathcal B_n}$. 
Its restriction on $\Lambda$ can be determined using Definition \eqref{df: GS}. 
Equivalently, it can be determined using the difference of energies which extends more easily to infinite volume.
More precisely, the restriction of the ground state to $\Lambda$ is the unique $\nu$-a.s.~configuration 
$\eta\in\{-1,+1\}^\Lambda$ such that
\begin{equation}
\label{eqn: GS diff}
H_{\Lambda, J}(\eta)-H_{\Lambda, J}(\eta')+ E_n(\eta,\eta')<0 \ \ \forall \eta'\neq \eta\ ,
\end{equation}
where 
\begin{equation}
\label{eqn: E}
E_n(\eta,\eta')= \sum_{e\in \mathcal B_n^*\setminus \Lambda^*}-J_e(\sigma^\eta_e-\sigma^{\eta'}_e) + \sum_{\{x,y\}\in \partial  \mathcal B_n} -J_{xy} (\sigma_x^\eta-\sigma_y^{\eta'}) \xi_y\ .
\end{equation}
The variable $E_n(\eta,\eta')$ is the difference of energies {\it outside $\Lambda$} of the states $\sigma^\eta$, $\sigma^{\eta'}$ that minimizes $H_{\mathcal B_n, J}$ over the configurations equal to $\eta$ on $\Lambda$, and similarly for $\eta'$.
The advantage of this formulation is two-fold. First, as detailed below, the random variables $E_n(\eta,\eta')$, $n\geq 1$, as a function of $J$ are tight. Second, the random variables $\vec E_n=\big(E_n(\eta,\eta');\eta, \eta'\in \{-1,+1\}^\Lambda\big)$ are independent of $J_\Lambda$.
This is because the restriction to  fixed $\eta$ cancels out the dependence on $J_\Lambda$.
These two observations lead to:
\begin{lem}
\label{lem: metastate}
Fix $\Lambda\subset \Z^d$.
There exists a subsequence such that the joint distribution of $(J_,\vec E_n)$ converges weakly 
to a probability measure on $(J, \vec{E})$ where $\vec{E}=\big(E(\eta, \eta'); \eta,\eta'\in\{-1,1\}^{\Lambda}\big)$ with the properties:
\begin{itemize}
\item {\rm Boundedness}: For every $\eta,\eta'$
\begin{equation}
\label{eqn: bound}
E(\eta,\eta')\leq \sum_{e\in \partial \Lambda}2 |J_e|\ \ a.s.
\end{equation}
\item {\rm Linear relations}: $E(\eta,\eta)=0$ for every $\eta$, and for every $\eta, \eta',\eta''$,
\begin{equation}
\label{eqn: linear}
E(\eta, \eta'')=E(\eta, \eta')+ E(\eta', \eta'')\ \ a.s.
\end{equation}
\item {\rm Independence}: Write $J=(J_{\Lambda},J_{\Lambda^c})$ where $J_{\Lambda}=(J_e, e\in \Lambda^*)$. Then
the pair $(J_{\Lambda^c},\vec E)$ is independent of $J_\Lambda$.
\end{itemize}
\end{lem}
\begin{proof}
The tightness of the random variables $(\vec E_n, n\in \N)$ follows from the inequality
\begin{equation}
\label{eqn: bound n}
E_n(\eta,\eta')\leq \sum_{e\in \partial \Lambda}2 |J_e|\ .
\end{equation}
This is because
$$
H_{\mathcal B_n, J}(\sigma^\eta)-H_{\mathcal B_n, J}(\sigma^{\eta'})
\leq H_{\Lambda, J}(\eta)-H_{\Lambda, J}(\eta') +\sum_{e\in \partial \Lambda}2 |J_e|\ ,
$$
where the inequality is obtained by replacing $\sigma^\eta_x$ for $x\in \Lambda^c$ by $\sigma^{\eta'}_x$
in the difference of energies \eqref{eqn: GS diff} and \eqref{eqn: E},
which increases the energy by definition of $\sigma^\eta$.
The tightness of the pair $(J,\vec E_n)$ directly follows since the $J$'s are IID. 
Equation \eqref{eqn: bound} is also straightforward from Equation \eqref{eqn: bound n} at finite $n$.
The linear relations are satisfied for every $\vec E_n$ and therefore extend to the weak limits. 
The same holds for the independence with $J_\Lambda$.
\end{proof}

Since we are interested in incongruent states, only the values $\eta_e=\eta_x\eta_y$ on an edges $e$ matter.
With this in mind we consider the collection $\vec{E}=\big(E(\eta, \eta'); \eta,\eta'\big)$ as indexed by elements $\eta,\eta'\in \{-1,+1\}^{\Lambda^*}$ where $\eta_e=\eta_x\eta_y$.
Of course, if two spin configurations are equal up to a global spin flip, then they correspond to the same element in $\{-1,+1\}^{\Lambda^*}$.
We then choose as a representative the one with {\it smaller energy} $E$; that is, we pick $\eta$ if $E(\eta,\eta')<0$ and $\eta'$ if $E(\eta,\eta')>0$.
In the case where $E(\eta,\eta')=0$, which happens for example when periodic boundary conditions are considered, the $\eta$'s are simply identified.
\\

We write $\kappa_{J_{\Lambda^c}}(\rd \vec E)$ for the conditional distribution of $\vec{E}$ given $J$, highlighting the independence from $J_\Lambda$, constructed from Lemma \ref{lem: metastate}.
 The variable $\vec E$ retains the relevant information on the boundary condition to determine the ground state in the box $\Lambda$ (up to a global spin flip).
Given $\vec E$, the ground state  in $\Lambda$ can be determined uniquely as a function of $J_\Lambda$ as in \eqref{eqn: GS diff} among all configuration in $\{-1,1\}^{\Lambda^*}$, assuming there are no non-trivial degeneracies.
These degeneracies will occur, for a given $\vec E$ sampled from $\kappa_{J_{\Lambda^c}}$, on the {\it critical set} given by the union of hyperplanes
\begin{equation}
\label{eqn: C}
\mathcal C=\mathcal C(\vec E)=\bigcup_{\substack{\eta\neq \eta'}}\big\{J_\Lambda\in \R^{\Lambda^*}:\sum_{e\in \Lambda^*}J_e(\eta_e-\eta_e')=E(\eta, \eta')\big\}\ .
\end{equation}
The union is over distinct $\eta,\eta'\in \{-1,+1\}^{\Lambda^*}$.
We refer to each hyperplane defining the critical set as a {\it critical hyperplane}. 
We work out the details of the cases where $\Lambda$ contains one and two edges in Remark \ref{rem: examples} below.

On the complement of the critical sets, it is possible to order the spin configurations in $\Lambda$ (up to spin flips) in decreasing order of their energies. In particular, it is possible to determine the ground state.
\begin{prop}
\label{prop: ordering}
For a given $\vec E$ with the property \eqref{eqn: linear} and $J_\Lambda\in \R^{\Lambda^*}\setminus \mathcal C$, there is a well-defined ordering $\eta^{(1)}\prec\eta^{(2)}\prec\dots$ of the elements of $\{-1,+1\}^{\Lambda^*}$
given by
\begin{equation}
\label{eqn: diff}
\eta \prec \eta ' \Longleftrightarrow E(\eta,\eta')+H_{\Lambda,J}(\eta)-H_{\Lambda,J}(\eta')<0\ .
\end{equation}
\end{prop}
The critical set corresponds to the value of $J_\Lambda$ for which $\eta^{(i)}=\eta^{(j)}$ for some pair $i\neq j$.
\begin{proof}
As a reference point, take $\eta^0\in$ with $\eta_e^0=+1$ for all $e\in \Lambda^*$. 
If $J_\Lambda \notin \mathcal C$, there exists a unique $\eta$ that minimizes the difference of energy
$$
E(\eta, \eta_0)+H_{\Lambda,J}(\eta)-H_{\Lambda,J}(\eta_0)\ .
$$
Indeed, if $\eta'\neq \eta$ was also a minimizer we would have by the linearity \eqref{eqn: linear} that
$E(\eta, \eta')+H_{\Lambda,J}(\eta)-H_{\Lambda,J}(\eta')=0$ contradicting the fact that $J_\Lambda$ is not in $\mathcal C$. 
Denote this unique minimizer by $\eta^{(1)}$. 
We define $\eta^{(2)}$ as the minimizer of the difference of energy $E(\eta, \eta^{(1)})+H_{\Lambda,J_\Lambda}(\eta)-H_{\Lambda,J_\Lambda}( \eta^{(1)})$ over $\eta$'s not equal to $\eta^{(1)}$. By construction this difference of energy is strictly positive.
Again $\eta^{(2)}$ is uniquely defined by linearity. The whole sequence $\eta^{(j)}$ is constructed this way until $\{-1,+1\}^{\Lambda^*}$ is exhausted.
The relation $\eta\prec \eta'$ is straightforward from construction. 
\end{proof}

The ordering introduced above defines three important maps from $\R^{\Lambda^*}\setminus \mathcal C$ to $\{-1,1\}^{\Lambda^*}$ which allow the study of excitations as a local function of the couplings.
The {\it ground state map} is the map
\begin{equation}
\label{eqn: gs map}
\begin{aligned}
\sigma(\cdot):\R^{\Lambda^*}\setminus \mathcal C&\to \{-1,+1\}^{\Lambda^*}\\
J_\Lambda&\mapsto \sigma(J_\Lambda)=\eta^{(1)}
\end{aligned}
\end{equation}
where $\sigma(J_\Lambda)$ is $\eta^{(1)}$ in the ordering at $J_\Lambda$ given by Proposition \ref{prop: ordering}. 
For a given edge $b\in\{-1,+1\}^{\Lambda^*}$, we define the {\it excitation map at the edge $b$} as
\begin{equation}
\label{eqn: pm map}
\begin{aligned}
\sigma^{+,b}(\cdot):\R^{\Lambda^*}\setminus \mathcal C&\to \{-1,+1\}^{\Lambda^*}\\
J_\Lambda&\mapsto \sigma^{+,b}(J_\Lambda)
\end{aligned}
\end{equation}
 where $\sigma^{+,b}(J_\Lambda)\prec \eta$ for all $\eta\neq \sigma^{+,b}(J_\Lambda)$ with $\eta_b=+1$.
 In words, $\sigma^{+,b}(J_\Lambda)$ is the configuration of smallest energy with the restriction that $\eta_b=+1$.
The map $\sigma^{-,b}(\cdot)$ is defined similarly, but restricting to $\eta$'s with $\eta_b=-1$.
Note that we evidently have $\sigma(J_\Lambda)= \sigma^{+,b}(J_\Lambda)$ or  $\sigma(J_\Lambda)=\sigma^{-,b}(J_\Lambda)$.\\

The precise definition of $\kappa_J(\rd \sigma)$ appearing in Equation \eqref{eqn: prob} can now be given.
We use the same notation for both the measure on $\vec E$ and $\sigma$ as they are directly related.
\begin{df}
\label{df: GS prob}
The probability measure $\kappa_J(\rd \sigma)$ on infinite-volume ground states restricted to $\Lambda$ 
is the distribution of $\sigma(J_\Lambda)$ as defined in \eqref{eqn: gs map} under $\kappa_{J_{\Lambda^c}}(\rd \vec E)$. 
\end{df}

\begin{rem}
{\rm 
It is not hard to check that the definition of $\kappa_J(\rd \sigma)$ is equivalent to taking weak limits of the distribution of the ground states (up to a spin flip)
of $H_{\mathcal B_n, J}$ given in \eqref{eqn: H_n} as a probability measures on $\{-1,+1\}^{\Lambda^*}$. 
The construction in Lemma \ref{lem: metastate} has the disadvantage of having the dependence on $J_\Lambda$ implicit in $\sigma$, which makes impossible to study the local modification of the couplings.
The advantage of working with $\vec E$ is that the dependence on $J_\Lambda$ appears solely in the Hamiltonian in $\Lambda$ as in \eqref{eqn: diff}.
This property is sometimes referred to as {\it coupling covariance}, see e.g.~\cite{ANSW16}.
}
\end{rem}

\begin{rem}
\label{rem: examples}
{\rm 
The simplest cases of excitation metastates where $\Lambda$ contains one and two edges were worked out in \cite{NS01} and \cite{ADNS10} respectively.
We briefly recall these examples here to illustrate the general theory.

{\it Case of one edge}. Consider $\Lambda=\{x,y\}$ where $x,y$ are nearest-neighbor vertices with $b=\{x,y\}$. 
We have that $\eta_b=+1$ or $-1$. The collection $\vec E$ of energies has  four values $E(+,-), 0,0$ and $E(-,+)=-E(+,-)$. 
The critical set is defined by a single equation:
$$
2J_b=E(+,-)\ ,
$$
and consists of the {\it critical value} $\mathcal C_b=E(+,-)/2$. Note that $\mathcal C_b$ is independent of $J_b$ by Lemma \ref{lem: metastate}.
The ground state $\sigma(J_b)$ at the edge $b$ is $+1$ for $J_b>\mathcal C_b$ and $-1$ for $J_b<\mathcal C_b$. 
The flexibility of the edge $b$, defined in \eqref{eqn: flex} below, is the function $F_b(J_b)$ giving the energy difference between $\sigma^{+,b}$ and $\sigma^{-,b}$ in absolute value. Here it
is simply
$$
F_b(J_b)=|2J_b-E(+,-)|=2|J_b-\mathcal C_b|\ .
$$

\medskip

{\it Case of two edges}. Take $\Lambda=\{x,y,w,z\}$ with edges $b=\{x,y\}$ and $e=\{w,z\}$. In this case, the configuration $\eta$ takes value in $\{++;+-;-+;--\}$ 
where we write the configuration at $b$ first and at $e$ second. The critical set is defined by six equations
\begin{equation}
\label{eqn: ex two}
\begin{aligned}
2J_b&=E(++;-+) \qquad  &2J_b=E(+-;--) \\
2J_e&=E(++;+-) \qquad  &2J_e=E(-+,--) \\
2(J_b+J_e)&=E(++,--) \qquad  &2(J_b-J_e)=E(+-;-+)\ .
\end{aligned}
\end{equation}
There are three possible scenarios: $E(++,-+)>E(+-;--)$, $E({\tiny ++},-+)=E(+-;--)$, $E(++,-+)<E(+-;--)$.
We look at the first case. It is depicted in Figure \ref{fig: flex} in the $(J_b,J_e)$-plane. 
Note that by linearity \eqref{eqn: linear} the inequality implies also $E(++;+-)>E(-+;--)$ by adding $E(-+,+-)$ on both sides. 
The equations \eqref{eqn: ex two} define sixteen regions where the ordering of the $\eta$'s (in terms of the energy differences) is non-degenerate.
(Not all twenty-four orderings of the four states are possible, since some are precluded by the energies.)

We now focus on the degeneracy of the ground state.
This happens at points $(J_b,J_e)$ where the energy difference between $\sigma^{+,b}(J_b,J_e)$ and $\sigma^{-,b}(J_b,J_e)$, or between $\sigma^{+,e}(J_b,J_e)$ and $\sigma^{-,e}(J_b,J_e)$, is zero. We treat the first case.
The state $\sigma^{+,b}(J_b,J_e)$ can be either $(++)$ or $(+-)$, and $\sigma^{-,b}(J_b,J_e)$ can be either $(-+)$ or $(--)$.
The energy difference between each is
$$
E(++;+-)-2J_e\qquad E(-+;--)-2J_e\ .
$$
Both are negative for $J_e$ large enough, showing that we must have $\sigma^{+,b}(J_b,J_e)=(++)$  and $\sigma^{-,b}(J_b,J_e)=(-+)$.
The same way we have $\sigma^{+,b}(J_b,J_e)=(+-)$ and $\sigma^{-,b}(J_b,J_e)=(--)$ for $J_e$ small enough. 
We conclude that the ground state degeneracy occurs at $J_b=E(++;-+)/2$ for $J_e$ large enough, and at $J_b=E(+-;--)/2$ for $J_e$ small enough. There is also a middle region where the degeneracy occurs between $(+-)$ and $(-+)$ at $J_b=J_e+E(+-;-+)/2$. 
\begin{figure}[h]
\includegraphics[height=7cm]{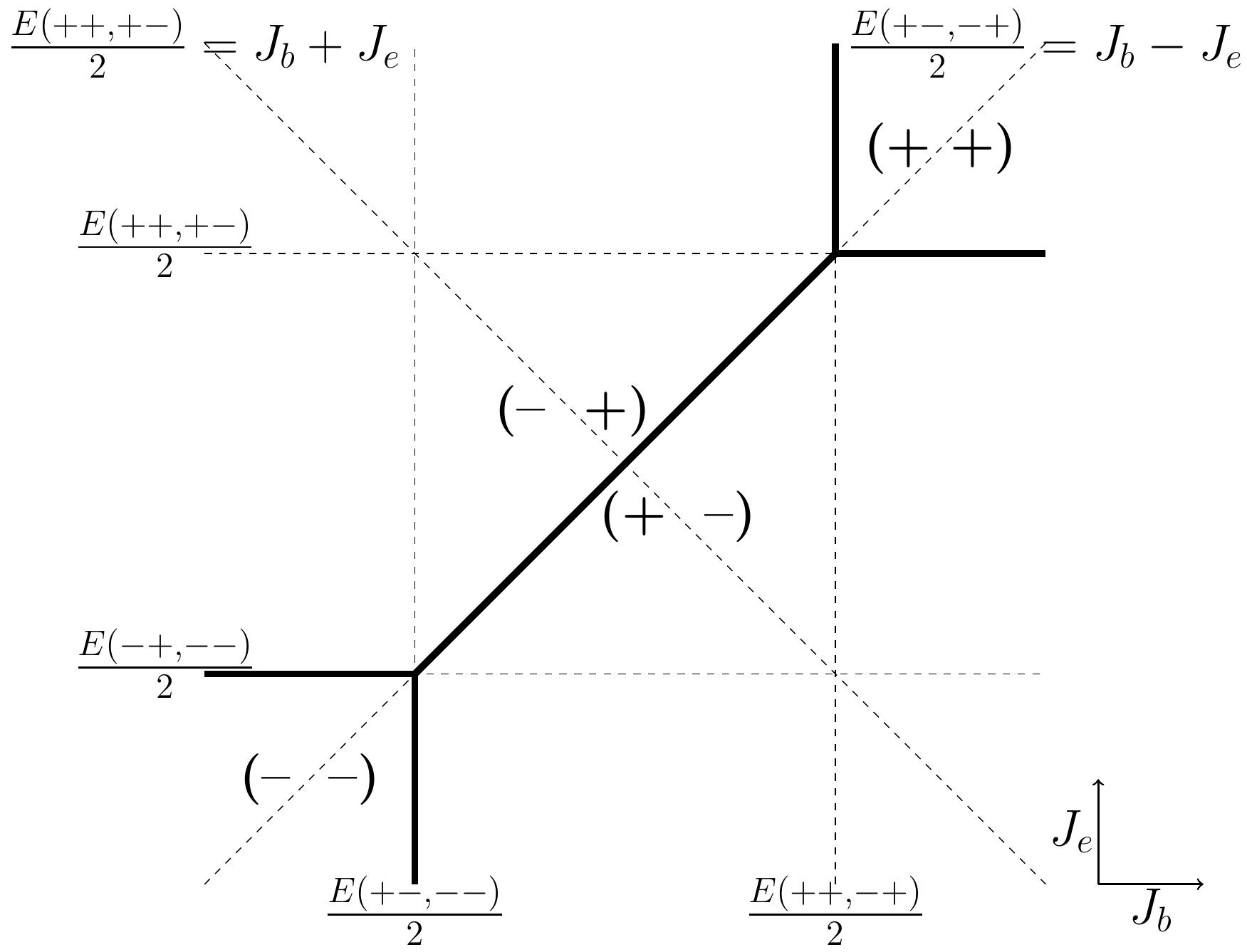}
\caption{An illustration of the critical set for two edges $b$ and $e$ in the $(J_b,J_e)$-plane. The dotted lines are the lines where the energy difference between two states is zero. The bold lines represent a degeneracy of the ground state. 
They delimit four regions where the ground state  is non-degenerate.
}
\label{fig: flex}
\end{figure}
}
\end{rem}

\subsection{Critical Droplets and Flexibilities}
Now that we can control the ground state as a function of $J_\Lambda$, we can study how the ground state at given edge $b$ depends on the couplings in $\Lambda$. The ground state configuration at $b$ is $+1$ if $\sigma^{+,b}(J_\Lambda)$ is the ground state and $-1$ if $\sigma^{-,b}(J_\Lambda)$ is the ground state.
The difference of energies between the two determines the correct value. Changes in the ground state occur when this energy difference is zero.
With this in mind, we consider the absolute value of the difference of energies or {\it  flexibility of the edge $b$} introduced in \cite{NS2000}:
\begin{equation}
\label{eqn: flex}
\begin{aligned}
F_b(J_\Lambda)&=\left|-\sum_{e\in \Lambda^*} J_e \big(\sigma^{+,b}_e(J_\Lambda)-\sigma^{-,b}_e(J_\Lambda)\big)+ E(\sigma^{+,b}(J_\Lambda), \sigma^{-,b}(J_\Lambda))\right|
\ .
\end{aligned}
\end{equation}
The flexibility $F_b$ is a map that measures the sensitivity of the ground state at the edge $b$ as a function of the couplings,
as highlighted in Proposition \ref{prop: criterion}.
The terms in the first sum are only non-zero on the edges of the {\it boundary of the critical droplet of the edge $b$ in $\Lambda$} at $J_\Lambda\in \R^{\Lambda^*}\setminus \mathcal C$, defined to be the set
\begin{equation}
\label{eqn: droplet}
\partial\mathcal D_b(J_\Lambda)=\{e\in \Lambda^*: \sigma_e^{+,b}(J_\Lambda)\neq \sigma_e^{-,b}(J_\Lambda)\}\ .
\end{equation}
\begin{figure}
\includegraphics[height=6cm]{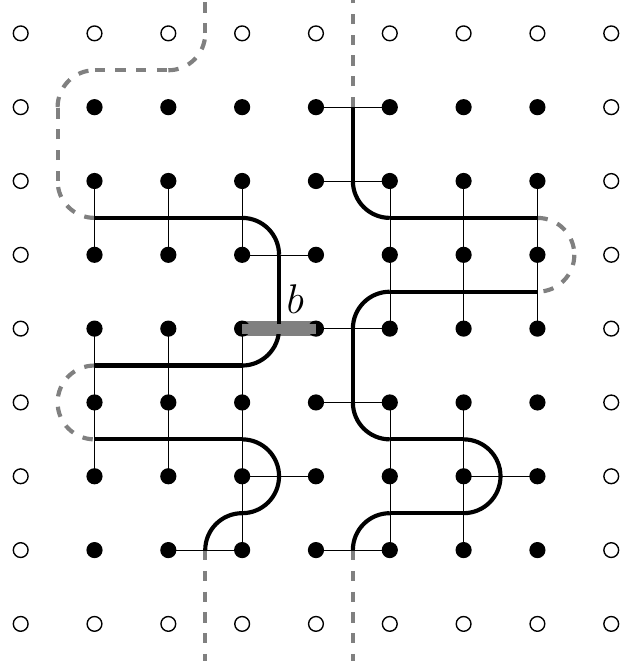}
\caption{An illustration of the critical droplet of an edge $b$ (in gray) and its boundary in $\Lambda$. The vertices in the box $\Lambda$ are black. The edges in $\partial\mathcal D_b(J_\Lambda)$ are the ones in $\Lambda^*$ that cross the boundary of the droplet.}
\label{fig: droplet}
\end{figure}

The following lemma is important to control the stability of the ground states as couplings are modified.
It shows that the flexibility uniquely extends to a continuous map on $\R^{\Lambda^*}$. 
\begin{lem}
\label{lem: continuity}
For every edge $b\in \Lambda^*$,
the map $J_\Lambda\mapsto F_b(J_\Lambda)$ on $\mathcal \R^{\Lambda^*}\setminus \mathcal C$ is a piecewise affine function with 
\begin{equation}
\frac{\partial F_b}{\partial J_e} (J_\Lambda)= 
\begin{cases}
2\sigma_e(J_\Lambda) & \text{ if $e\in \partial\mathcal D_b(J_\Lambda)$}\\
0 & \text{ otherwise.}
\end{cases}
\end{equation}
Furthermore, the map extends uniquely to a continuous function on $\R^{\Lambda^*}$.
\end{lem}

\begin{proof}
Consider $J_\Lambda \in \R^{\Lambda^*}\setminus \mathcal C$. By the definition of $\sigma(J_\Lambda)$ and the fact that $\R^{\Lambda^*}\setminus \mathcal C$ is open, we have that the map $\sigma(\cdot)$ is constant in a neighborhood $V=V(J_\Lambda)$ of $J_\Lambda$, and so are $\sigma^{+,b}(\cdot)$ and $\sigma^{+,b}(\cdot)$. Write $\sigma$, $\sigma^{+}$, $\sigma^-$ for the respective values in $V$. Suppose without loss of generality, that $\sigma=\sigma^{+}$. The critical droplet boundary $\partial\mathcal D_b$ is also constant on $V$.
 The flexibility of the edge $b$ on $V$ takes the form
$$
F_b(y)=\sum_{e\in \partial\mathcal D_b} 2y_e\sigma_e - E(\sigma^{+}, \sigma^{-})\ , \  y \in V\ .
$$
Therefore, the derivative in the $y_e$-direction equals $2\sigma_e$ if $e\in \partial\mathcal D_b$ and is $0$ otherwise as claimed.
The fact that $F_b$ is a piecewise affine function on $\mathcal \R^{\Lambda^*}\setminus \mathcal C$ follows from the form of the derivatives and the fact that $\sigma$ is piecewise constant.

It remains to prove the extension to a unique continuous function.
Take $J_\Lambda\in \mathcal C$. By the same reasoning as above, the function $F_b$ is well-defined and continuous at $J_\Lambda$  unless 
 there are degeneracies in the definition of $\sigma^{+,b}(J_\Lambda)$ or $\sigma^{-,b}(J_\Lambda)$.
This happens if there are more than one minimizer for the difference of energy \eqref{eqn: diff} among the configurations with $+1$ at the edge $b$, and the ones with $-1$ at the edge $b$. 

Suppose there is exactly one degeneracy for the minimizer $\sigma^{+,b}$ at $J_\Lambda$.
This means that at $J_\Lambda$ there are configurations $\eta^+$ and $\tilde \eta^+$ such that
\begin{equation}
\label{eqn: equality}
\sum_{e\in \Lambda^*}J_e(\eta^+_e-  \tilde\eta^+_e) =E(\eta^+, \tilde \eta^+) \ .
\end{equation}
In particular, this means that $J_\Lambda$ sits on the hyperplane defined by $\eta^+,\tilde \eta^+$. 
Note that on this hyperplane the expressions for the flexibility in \eqref{eqn: flex} for $\eta^+$ and $\tilde \eta^+$ agree since by \eqref{eqn: linear} and by \eqref{eqn: equality}
\begin{equation}
\label{eqn: flex equal}
\begin{aligned}
&-\sum_{e\in \Lambda^*}J_e(\eta^+_e-  \eta^-_e) +E(\eta^+, \eta^-) \\
&=-\sum_{e\in \Lambda^*}J_e(\tilde\eta^+_e-  \eta^-_e) +E(\tilde\eta^+, \eta^-) 
-\sum_{e\in \Lambda^*}J_e(\eta^+_e-  \tilde\eta^+_e) +E(\eta^+, \tilde \eta^+)\\
&=-\sum_{e\in \Lambda^*}J_e(\tilde\eta^+_e-  \eta^-_e) +E(\tilde\eta^+, \eta^-)\ .
\end{aligned}
\end{equation}
This implies that the choice of representative for $\sigma^{+,b}$ on the hyperplane is irrelevant as far as the flexibility is concerned.
Therefore the flexibility extends continuously on the hyperplane.

Now suppose that there is more than one degeneracy for $\sigma^{+,b}$ or for $\sigma^{-,b}$ at $J_\Lambda$.
Without loss of generality suppose that there are $m$ configurations $\eta^1,\dots, \eta^m$ with $+1$ at the edge $b$ with the same energy difference. 
(The reasoning for degeneracies for the $-1$ excitation is the same.) 
Then by definition this is the same as having the relations
\begin{equation}
\label{eqn: equiv}
\sum_{e\in \Lambda^*}J_e(\eta^i_e-  \eta^j_e) =E(\eta^i,  \eta^j) \qquad \text{ for all $i,j\leq m$.}
\end{equation}
In other words, $J_\Lambda$ lies at the intersection of the hyperplanes defined by the $\eta^i$'s.
On each hyperplane, the flexibility is well-defined and continuous as shown above. 
Moreover, the same reasoning as in \eqref{eqn: flex equal} shows that these definitions must agree on the intersection by the relations \eqref{eqn: equiv}.
This concludes the proof of the lemma.
\end{proof}

We now study the stability of ground states as couplings in $\Lambda$ are varied.
For this, we fix $J_\Lambda, J_{\Lambda}'\in \R^{\Lambda^*}$ and consider the curve given by the non-linear interpolation
\begin{equation}
\label{eqn: line}
J_{\Lambda}(t)=\ee^{-t}J_\Lambda+\sqrt{1-\ee^{-2t}}J'_\Lambda,\ \  t\geq0\ .
\end{equation}

\begin{lem}
\label{lem: number of planes}
Consider the curve $J_\Lambda(t)$, $t\geq 0$ defined in \eqref{eqn: line}.
The number of $t$'s such that $J_\Lambda(t)$ is in the critical set is smaller than $4^{|\Lambda^*|}$.
\end{lem}
\begin{proof}
A given critical hyperplane is determined by a point $y=(y_e,e\in\Lambda^*)$ on the hyperplane and a vector $v=(v_e,e\in\Lambda^*)$ orthogonal to it.
If $J_\Lambda(t)$ intersects the hyperplane at $t$, then $t$ must satisfy the equation
$$
\sum_{e\in\Lambda^*}v_eJ_e(t)=v_ey_e\ .
$$
By writing the expression for $J_e(t)$, this yields an equation of the following form for $t$:
$$
a\ee^{-t}+b\sqrt{1-\ee^{-2t}}=c\ ,
$$
where $a,b,c$ depend on $J_\Lambda, J_\Lambda', v, y$. This equation has at most two solutions.
Since there are at most $2^{\Lambda^*}\cdot 2^{\Lambda^*-1}$ hyperplanes, we obtain the claimed bound.
\end{proof}

For given endpoints $(J_\Lambda, J_\Lambda')$ for the curve \eqref{eqn: line}, we write $F_b(t)=F_b(J_\Lambda(t))$, $\sigma^{\pm, b}(J_\Lambda(t))=\sigma^{\pm, b}(t)$, and $\sigma(J_\Lambda(t))=\sigma(t)$
for simplicity. 
The following result gives a criterion for the stability of the ground state at an edge in terms of its flexibility.
In short, the ground state remains the same as the couplings in $\Lambda$ are varied as long as the flexibility is not $0$. 
\begin{prop}
\label{prop: criterion}
Consider $b\in \Lambda^*$ and the curve $t\mapsto J_\Lambda(t)$ defined in \eqref{eqn: line}. For $\nu$-almost all $(J_{\Lambda}, J_\Lambda')$, we have the following implication:
$$
\text{ if $F_b(s)> 0\ \forall 0\leq s\leq t$, then $\sigma_b(s)=\sigma_b(0),\ \forall 0\leq s\leq t\ .$}
$$
\end{prop}
\begin{proof}
First, observe that since the curve $J_\Lambda(t)$ intersects $\mathcal C$ finitely many times by Lemma \ref{lem: number of planes}, the limits $\lim_{t\downarrow t_0} \sigma(t)$ and  $\lim_{t\uparrow t_0} \sigma(t)$ must be well-defined. 
Suppose there exists $t_0>0$ such that $\lim_{t \downarrow t_0}\sigma_b(t)=+1$ and $\lim_{t \uparrow t_0}\sigma_b(t)=-1$ (or vice-versa).
Then $t_0$ must belong to $\mathcal C$. 
Denote the two limits  $\lim_{t\downarrow t_0} \sigma(t)$ and  $\lim_{t\uparrow t_0} \sigma(t)$ by $\sigma^+$ and $\sigma^-$ respectively.
The excitations $\sigma^{+,b}$ and $\sigma^{-,b}$ might be degenerate at $t_0$. 
But by the continuity proved in Lemma \ref{lem: continuity}, the flexibility is independent of the choice of the representatives for $\sigma^{+,b}$ and $\sigma^{-,b}$. 
We pick $\sigma^+$ and $\sigma^-$ for representatives. This means that the flexibility at $t_0$ can be written in two ways using $\sigma^+$ and $\sigma^-$:
$$
\begin{aligned}
 E(\sigma^{+},\sigma^{-})-\sum_{e}J_e(t_0)(\sigma^{+}_e-\sigma^{-}_e)=\lim_{t\uparrow t_0}F_b(t)=
\lim_{t\downarrow t_0}F_b(t)&= E(\sigma^{-},\sigma^{+})-\sum_{e}J_e(t_0)(\sigma^{-}_e-\sigma^{+}_e)\ .
\end{aligned}
$$
Since one is the negative of the other (note that $E(\eta,\eta')=-E(\eta',\eta)$ by \eqref{eqn: linear}), we conclude that $F_b(t_0)=0$ as claimed.
\end{proof}

\begin{prop}
\label{prop: flex diff}
Consider $b\in \Lambda^*$ and the curve $s\mapsto J_\Lambda(s)$ defined in \eqref{eqn: line}. 
We have for all $0\leq t\leq 1$ that
$$
\big|F_b(t)-F_b(0)\big| \leq 6 \sqrt{t} \cdot \max_{e\in \Lambda^*}(|J_e|\vee |J_e'|)\cdot \max_{s\leq t}|\partial D_b(s)|\ .
$$
\end{prop}

\begin{proof}
Let $K$ be the number of critical hyperplanes crossed by $J_\Lambda(s)$ before time $t$. By Lemma \ref{lem: number of planes}, this number is less than $4^{|\Lambda^*|}$.
Moreover, if we denote by $t_k$, $k\leq K$, the values at which the curve intersects $\mathcal C$, we must have that it intersects exactly one hyperplane almost surely by the same lemma. 
This means that the maps $s\mapsto\sigma(s)$ and $s\mapsto\sigma^{\pm,b}(s)$ (and in particular the critical droplet $\mathcal D_b(s)$) are well-defined and constant on each interval $(t_k,t_{k+1})$. 
By the continuity of the flexibility in Lemma \ref{lem: continuity}, it is therefore possible to expand $F_b(t)$ as follows
\begin{equation}
\label{eqn: decomp}
\begin{aligned}
F_b(t)-F_b(0)
&=\sum_{k: t_k<t}\int_{t_k}^{t_{k+1}\wedge t}  \nabla F_b(s) \cdot \frac{\rd J_\Lambda}{\rd s} (s) \rd s\\
&=\sum_{k: t_k<t}\sum_{e\in\partial\mathcal D_b(k)}2\sigma_e(k) \{J_e(t_{k+1}\wedge t)- J_e(t_{k})\}\ ,
\end{aligned}
\end{equation}
where we used the gradient in Lemma \ref{lem: continuity}. 
The notation $\partial\mathcal D_b(k)$ stands for $\partial\mathcal D_b(s)$ when $s\in (t_k,t_{k+1})$, and similarly for $\sigma_e(k)$.
Note that
$$
\begin{aligned}
|J_e(t_{k+1})-J_e(t_{k})|
&=|J_e|(\ee^{-t_{k}}-\ee^{-t_{k+1}})+|J'_e|(\sqrt{1-\ee^{-2t_{k+1}}}-\sqrt{1-\ee^{-2t_{k}}})\\
&\leq \max_{e\in \Lambda^*}(|J_e|\vee |J_e'|) \cdot \Big(\ee^{-t_{k}}-\ee^{-t_{k+1}}+\sqrt{1-\ee^{-2t_{k+1}}}-\sqrt{1-\ee^{-2t_{k}}}\Big)\ .
\end{aligned}
$$
Putting this estimate back in \eqref{eqn: decomp} yields
$$
|F_b(t)-F_b(0)|
\leq    2\max_{e\in \Lambda^*}(|J_e|\vee |J_e'|)\cdot \max_{s\leq t}|\partial D_b(s)|\cdot(1-\ee^{-t} + \sqrt{1-\ee^{-2t}})\ .
$$
The final estimate follows from the fact that $1-\ee^{-x}\leq x$ for $x\geq 0$, and $t+\sqrt{2t}\leq 3\sqrt{t}$ for $0\leq t\leq1$.

\end{proof}

\section{A Variance Bound for Gaussian Couplings}
\label{sect: lemma}
In this section, we prove variance bounds using the local modification of couplings described in Section \ref{sect: metastate}.
The main result is the proof of Theorem \ref{thm: main} relating the existence of incongruent states and disorder chaos.
The following result is standard, see e.g. \cite{AT07,C14}. We prove it for completeness.
\begin{lem}
\label{lem: gaussian variance}
Let $Y=(Y_i,i\leq n)$ and $Y'=(Y'_i,i\leq n)$  be two independent copies of a $n$-dimensional Gaussian vector.
Consider $h: \R^n\to \R$ in $\mathcal C^2(\R^n)$ with bounded derivatives. We have
\begin{equation}
\var (h(Y))=\int_0^{\infty} \sum_{i\leq n}\E\left[\partial_ih(Y)\cdot\partial_ih(Y(s))\right] \ee^{-s}\rd s\ ,
\end{equation}
where $Y(s)=\ee^{-s}Y+\sqrt{1-\ee^{-2s}}Y'$. 
In particular, for any $t\geq 0$,
\begin{equation}
\label{eqn: bound pos}
\var (h(X))\geq \int_0^{t} \sum_{i\leq n}\E\left[\partial_ih(Y)\cdot\partial_ih(Y(s))\right] \ee^{-s}\rd s\ .
\end{equation}
\end{lem}
\begin{proof}
Consider the $(2n)$-dimensional Gaussian vector $X(t)=\ee^{-t}(Y,Y)+\sqrt{1-\ee^{-2t}}(Y',Y'')$ where $Y''$ is yet another independent copy of $Y$.
Write $X_A=X_A(s)=\ee^{-s}Y+\sqrt{1-\ee^{-2s}}Y'$ for the first $n$ component of $X(t)$, and $X_B=X_B(s)=\ee^{-s}Y+\sqrt{1-\ee^{-2s}}Y''$ for the $n$ last. 
It is clear that
\begin{equation}
\label{eqn: var1}
\var (h(X))=\int_0^{\infty} -\frac{\rd }{\rd s}\E[h(X_A)h(X_B)] \rd s .
\end{equation}
Gaussian integration by parts implies that for a function $g: \R^{2n}\to \R$ of moderate growth 
and two independent, but not identically distributed, $2n$-dimensional vectors  $Z$ and $Z'$,
$$
\frac{\rd }{\rd u}\E[g(Z(u))]=\frac{1}{2}
\sum_{i,j=1}^{2n} \left(\E[Z_iZ_j]-\E[Z'_iZ'_j]\right)\E[\partial_i\partial_j g(Z(u))]\ ,
$$
for $Z(u)=\sqrt{u}Z+\sqrt{1-u} Z'$, see e.g.~\cite{AT07}.
We apply this with $Z=(Y,Y)$, $Z'=(Y',Y'')$ and $g(Z(u))=h(\sqrt{u}Y+\sqrt{1-u} Y')\cdot h(\sqrt{u}Y+\sqrt{1-u} Y'')$. 
In this instance, by independence, we have $\E[Z_iZ_j]=\E[Z'_iZ'_j]=0$ unless $i=j$, $i=j+n$ or $j=i+n$. The case $i=j$ gives $\E[Z_iZ_j]-\E[Z'_iZ'_j]=0$ so only the two others gives a non-zero contribution with $\E[Z_iZ_j]-\E[Z'_iZ'_j]=\E[Z_iZ_j]=1$.
The derivatives in both cases $i=j+n$ and $j=i+n$ are
$$
\E[\partial_i\partial_j g(Z(u))]=\E[\partial_ih(X_A)\cdot \partial_j h(X_B)]\ .
$$
Putting this back in \eqref{eqn: var1} with $u=\ee^{-2s}$ yields
$$
\var (h(X))=\int_0^{\infty} \sum_{i\leq n}\E[\partial_ih(X_A)\cdot \partial_ih(X_B)] \ 2\ee^{-2s}\rd s
$$
since $\frac{\rd }{\rd u}=-2 \ee^{-2s}\frac{\rd }{\rd s}$. 
Observe that the joint distribution of $(X_A,X_B)$ is the same as $(Y,Y(t))$. 
The first claim then follows by the change of variable $s\to 2s$.
The second claim is straightforward from the fact that the term $\E[\partial_ih(X_A)\cdot \partial_i h(X_B)]$ is non-negative as can be seen by conditioning on $Y$.
\end{proof}

Recall the definition of the ground state map \eqref{eqn: gs map}.
As given in Definition \ref{df: GS prob}, the variance of $H_{\Lambda,\J}(\sigma^1)-H_{\Lambda,\J}(\sigma^2)$ under 
$\rd \PP= \rd \nu (J)\times \rd \nu (J') \times \rd \kappa^1_{J}(\sigma^1)\times  \rd \kappa^2_{J}(\sigma^2)$
 is equal to the variance of $H_{\Lambda,\J}(\sigma^1(J_\Lambda))-H_{\Lambda,\J}(\sigma^2(J_\Lambda))$
under the measure
\begin{equation}
\label{eqn: prob2}
\rd \PP= \rd \nu (J)\times  \rd \nu (J') \times  \rd \kappa^1_{J_{\Lambda^c}}(\vec E^1)\times  \rd \kappa^2_{J_{\Lambda^c}}(\vec E^2)
\end{equation}
We consider $J_\Lambda(t)$ as in Equation \eqref{eqn: line}, and $\sigma(J_{\Lambda}(t))=\sigma(t)$ for short.
\begin{lem}
\label{lem: variance}
Consider $\Lambda \subset \Z^d$ finite. 
We have for every $t\geq 0$,
\begin{equation}
\label{eqn: variance}
\var\Big(H_{\Lambda,\J}(\sigma^1(0))-H_{\Lambda,\J}(\sigma^2(0) )\Big)
\geq \int_0^t \sum_{b\in \Lambda^*}\E\Big[\big(\sigma^1_b(s)-\sigma^2_b(s)\big)\cdot\big(\sigma^1_b(0)-\sigma^2_b(0)\big) \Big]\ee^{-s}\rd s\ .
\end{equation}
\end{lem}
\begin{proof}
By conditioning on $(J_{\Lambda_c}, \vec E)$ we get by the conditional variance formula
$$
\var\Big(H_{\Lambda,\J}(\sigma^1(J_\Lambda))-H_{\Lambda,\J}(\sigma^2(J_\Lambda)) \Big)
\geq  \E\left[\var\Big(H_{\Lambda,\J}(\sigma^1(J_\Lambda))-H_{\Lambda,\J}(\sigma^2(J_\Lambda)) \Big| J_{\Lambda^c}, \vec E^1,\vec E^2\Big)\right]\ .
$$
The distribution of $J_{\Lambda}$ conditioned on $(J_{\Lambda_c}, \vec E^1,\vec E^2)$ remains IID Gaussian  by the independence in Lemma \ref{lem: metastate}.
We apply Lemma \ref{lem: gaussian variance} with $Y=J_\Lambda$ and $Y(t)=J_\Lambda(t)$.
To compute the derivatives, we used Proposition \ref{prop: ordering} and the definition of the ground state map \eqref{eqn: gs map}. Since the ground state $\sigma(J_\Lambda)$ is constant and well-defined on a set of full measure, the derivative $\partial_{J_b} \sigma^1_e(J_\Lambda)$ is $0$ $\nu$-a.s for every edge $e$.
Therefore we have
$$
\frac{\partial}{\partial J_b}\{H_{\Lambda,\J}(\sigma^1(J_\Lambda))-H_{\Lambda,\J}(\sigma^2(J_\Lambda))\}
=-(\sigma_b^1(J_{\Lambda})-\sigma_b^2(J_{\Lambda}))\ \ \nu-a.s.
$$
We conclude that 
$$
\begin{aligned}
&\E\left[\var\Big(H_{\Lambda,\J}(\sigma^1(J_\Lambda))-H_{\Lambda,\J}(\sigma^2(J_\Lambda)) \Big| J_{\Lambda^c}, \vec E^1, \vec E^2\Big)\right]\\
&= \sum_{b\in \Lambda^*}\int_0^\infty
 \E\left[(\sigma_b^1(J_\Lambda)-\sigma_b^2(J_\Lambda))(\sigma_b^1(J_\Lambda(s))-\sigma_b^2(J_\Lambda(s))\right] e^{-s}\rd s\ .
 \end{aligned}
$$
The lower bound restricted $t\geq 0$ follows from \eqref{eqn: bound pos}. The restriction to one edge $b$ holds for the same reason since the integrand is positive.
\end{proof}

\begin{proof}[Proof of Theorem \ref{thm: main}]
The theorem is an elementary consequence of Lemma \ref{lem: variance}.
First observe that the quantity
$$
\left(2-2Q_\Lambda(\sigma,\sigma')\right)^{1/2}=\frac{1}{|\Lambda^*|^{1/2}}\left(\sum_{b\in \Lambda^*} (\sigma_b- \sigma_b')^2\right)^{1/2}
=:\|\sigma-\sigma'\|
$$
satisfies the triangle inequality. In particular, we have
\begin{equation}
\|\sigma-\sigma'\|\geq \Big|\|\sigma-\sigma''\|-\|\sigma'-\sigma''\|\Big|\ .
\end{equation}
This inequality implies
$$
\begin{aligned}
Q_\Lambda(\sigma,\sigma')-Q_\Lambda(\sigma',\sigma'')
&=1-Q_\Lambda(\sigma',\sigma'')-1+Q_\Lambda(\sigma,\sigma')\\
&\geq \frac{1}{2}\left(\Big|\|\sigma-\sigma'\|-\|\sigma-\sigma''\|\Big|\right)^2-\frac{1}{2} \|\sigma-\sigma'\|^2\\
&=\frac{1}{2}\|\sigma-\sigma''\|^2-\|\sigma-\sigma'\|\|\sigma-\sigma''\|\\
&\geq \frac{1}{2}\|\sigma-\sigma''\|^2-2\|\sigma-\sigma'\|\ ,
\end{aligned}
$$
since $\|\sigma-\sigma''\|\leq 2$. We apply this inequality to $\sigma=\sigma^1(0)$, $\sigma'=\sigma^1(s)$, $\sigma''=\sigma^2(0)$ (and again with $1$ replaced by $2$) to rewrite the integrand in \eqref{eqn: variance} as
$$
|\Lambda^*|\cdot  \E\Big[\|\sigma^1(0)-\sigma^2(0)\|^2-2\sum_{i=1,2}\|\sigma^i(0)-\sigma^i(s)\|\Big]\ .
$$

By putting this back in \eqref{eqn: variance}, we have
$$
\begin{aligned}
&\frac{1}{|\Lambda^*|}\var\Big(H_{\Lambda,\J}(\sigma^1(J_\Lambda))-H_{\Lambda,\J}(\sigma^2(J_\Lambda)) \Big)\\
&\geq  \int_0^t\left\{2\E\Big[1-Q_\Lambda(\sigma^1(0),\sigma^2(0))\Big]-2\sqrt{2}\sum_{i=1,2} \E\Big[\Big(1-Q_\Lambda(\sigma^i(0),\sigma^i(s))\Big)^{1/2}\Big]\right\}\ee^{-s}\rd s\ .
\end{aligned}
$$

The claim follows by applying Jensen's inequality to the second term.
\end{proof}

\section{Disorder Chaos and Critical Droplets}
We start by establishing Corollary \ref{cor: ADC-> variance} as an elementary consequence of Theorem  \ref{thm: main}.
\begin{proof}[Proof of Corollary \ref{cor: ADC-> variance}]
On one hand, by definition of disorder chaos at scale $\alpha$, we have that for every $\e>0$ there is $C(\e)>0$ and $A_\e$ with $\PP(A_\e)>1-\e$ such that for every $t\leq C|\Lambda|^{-\alpha}$ and $\Lambda$ large enough
\begin{equation}
\label{eqn: overlap estimate}
\E\big[1-Q_\Lambda(\sigma^1, \sigma^1(t))\big]
= \E\big[1-Q_\Lambda(\sigma^1, \sigma^1(t));A_\e\big]+ \E\big[1-Q_\Lambda(\sigma^1, \sigma^1(t));A_\e^c\big]
\leq \e (1-\e)+2\e\ .
\end{equation}
On the other hand, if there exist incongruent states with positive $\PP$-probability, we must have by Fatou's lemma
\begin{equation}
\label{eqn: fatou}
\liminf_{\Lambda\to\Z^d}\E\Big[1-Q_\Lambda(\sigma^1(0),\sigma^2(0))\Big]
\geq 1-\E\Big[\limsup_{\Lambda\to \Z^d}Q_\Lambda(\sigma^1(0),\sigma^2(0))\Big]>0\ .
\end{equation}
The result follows from Theorem \ref{thm: main} by taking $\e$ small enough and $\Lambda$ large enough so that 
the right-hand side of Equation \eqref{eqn: variance bound} is strictly greater than $0$ uniformly for $s\leq C|\Lambda|^{-\alpha}$. 
\end{proof}

To prove Theorem \ref{thm: ADC}, we need the existence of many edges on which a given ground state is not too sensitive.
Since the statements of Theorem \ref{thm: ADC} involve only one replica $\sigma^1$, we set for the rest of this section
$$
\rd \PP= \rd \nu (J)\times  \rd \nu (J') \times  \rd \kappa_{J_{\Lambda^c}}(\vec E)\ .
$$
\begin{lem}
\label{lem: flex}
For any $\e>0$, there exists $\delta=\delta(\e)$ (independent of $\Lambda$) and a subset $B_\e$ of $(J, \vec E)$ with $\PP(B_\e)>1-\e$ such that on $B_\e$
$$
\#\{b\in \Lambda^*: |F_b(J_\Lambda)|>\delta\}> (1-\e)|\Lambda^*|\ .
$$
\end{lem}
\begin{proof}
Since $\#\{b\in \Lambda^*: |F_b(J_\Lambda)|>\delta\}=|\Lambda^*|-\#\{b\in \Lambda^*: |F_b(J_\Lambda)|\leq \delta\}$,
it suffices to show that for given $\e>0$ there is a $\delta$ small enough such that
$$
\PP\left(\#\{b\in \Lambda^*: |F_b(J_\Lambda)|\leq \delta\}> \e |\Lambda^*|\right)<\e\ .
$$
Markov's inequality implies that 
$$
\begin{aligned}
\PP\left(\#\{b\in \Lambda^*: |F_b(J_\Lambda)|\leq \delta\}> \e |\Lambda^*|\right)
&\leq \frac{1}{\e |\Lambda^*|} \E[\#\{b\in \Lambda^*: |F_b(J_\Lambda)|\leq \delta\}]\\
&= \frac{1}{\e |\Lambda^*|} \sum_{b\in \Lambda^*}\PP(|F_b(J_\Lambda)|\leq \delta)\ .
\end{aligned}
$$
We show that $\PP(|F_b(J_\Lambda)|\leq \delta\})<c\delta$ (uniformly on the edges $b$) for some $c>0$. The claim then follows by taking $\delta=\e^2/c$.
The key observation is that conditioned on $(J_{\Lambda^c}, \vec E)$, the states $\sigma^{\pm,b}(J_\Lambda)$ are independent of $J_b$.
This is because the contribution of $J_b$ in the difference of energies on the right side of \eqref{eqn: diff} cancels
when we restrict to $\eta$'s with $\eta_b=+1$ (or $\eta_b=-1$).
In particular, this means that we can write the flexibility \eqref{eqn: flex} as
$$
F_b(J_\Lambda)=2|J_b- \mathcal C_b|\ ,
$$
where $\mathcal C_b$ is a measurable function that only depends on $\vec E$ and $(J_e; e\in \Lambda^*, e\neq b)$, see also Remark \ref{rem: examples}. 
Therefore, the variable $J_b$ is independent of $\mathcal C_b$ under $\PP$ (by independence in Lemma \ref{lem: metastate}), and has the standard Gaussian distribution.
This implies 
\begin{equation}
\label{eqn: flex estimate}
\PP(|F_b(J_\Lambda)|\leq \delta\})=\PP(|J_b- \mathcal C_b|\leq \delta)\leq \PP(|J_b|\leq \delta)\ .
\end{equation}
It remains to observe that $\nu\{|J_b|\leq \delta\}\leq 2\delta/\sqrt{2\pi}$ to finish the proof.
\end{proof}

We now have all the ingredients to prove Theorem \ref{thm: ADC}.
\begin{proof}[Proof of Theorem \ref{thm: ADC}]
Fix $\e>0$.
From the definition \ref{df: ADC}, we need to find $C=C(\e)$ and a subset $A_\e$ of $(J, J',\vec E)$ with $\PP(A_\e)>1-\e$ on which
$$
\#\{b\in \Lambda^*: \sigma_b(t)=\sigma_b(0),\ \forall t\leq C|\Lambda|^{-\alpha}\}>(1-\e)|\Lambda^*|\ .
$$
By Proposition \ref{prop: criterion}, this would follow if we find $C$ and $A_\e$ on which
$$
\#\{b\in \Lambda^*: F_b(t)>0,\ \forall t\leq C|\Lambda|^{-\alpha}\}>(1-\e)|\Lambda^*|\ .
$$
We write $B_\e$ for the event in Lemma \ref{lem: flex}. 


Consider the event $\widetilde B_\e=\{\max_{e\in \Lambda^*} (|J_e|\vee |J'_e|)<\widetilde C\sqrt{\log |\Lambda|}\}$. 
A standard argument using Gaussian estimates shows that there exists $\widetilde C=\widetilde C(\e)$ large enough such that
$\PP(\widetilde B_\e)>1-\e$. We take $A_\e= B_\e\cap \widetilde B_\e$. We have by construction $\PP(A_\e)>1-2\e$. 
From Proposition \ref{prop: flex diff} and Equation \eqref{eqn: droplet ass}, it follows that on $(1-\e)|\Lambda^*|$ edges 
\begin{equation}
\label{eqn: flex ineq}
\begin{aligned}
F_b(t)\geq F_b(0)- 6 \sqrt{t} \cdot \max_{e\in \Lambda^*}(|J_e|\vee |J_e'|)\cdot \max_{s\leq t}|\partial D_b(s)|
&\geq \delta - 6  \widetilde C \sqrt{t}\cdot \sqrt{\log |\Lambda|} \cdot C |\Lambda|^\gamma \ .
\end{aligned}
\end{equation}
Taking $\alpha >2\gamma$, we conclude that $F_b(t)>\delta/2$ for $t\leq (6 C \widetilde C\delta)^{-2} |\Lambda|^{-\alpha}$ and $\Lambda$ large enough.
This completes the proof of the theorem.
\end{proof}

\begin{rem}
\label{rem: estimate}
{\rm 
The inequality \eqref{eqn: flex ineq} is far from optimal in general as it does not take into account the dependence between the droplet $\mathcal D_b$ and the couplings $J_\Lambda$, $J_\Lambda'$. The droplet $\mathcal D_b$ is special as it optimizes the energy on its boundary.
Here we bounded the value of the couplings on the boundary in an elementary way by the size of the boundary times the maximal value of the couplings in the whole box. 
The factor $\log |\Lambda|$ we get from this procedure is one of the reason why we cannot handle the case $\alpha=\gamma$.
To improve the result, one would have to develop a better understanding of the delicate  connection between the geometry of the underlying lattice and the extreme statistics of the couplings.
}
\end{rem}
Similar ideas gives weaker uniform bounds for the variance.
\begin{proof}[Proof of Corollary \ref{cor: variance one}]
Note first that the assumption $\PP(\mathcal I)>0$ implies that there exist an edge $b\in \Lambda^*$ and $c>0$ (both independent of $\Lambda$) such that for $\Lambda$ large enough $\E[1-\sigma_b^1\sigma_b^2]>c$. 
This is because Equation \eqref{eqn: fatou} implies
$$
\liminf_{\Lambda\to \Z^d} \frac{1}{|\Lambda^*|}\sum_{b\in \Lambda^*} \E[1-\sigma_b^1\sigma_b^2]>0\ .
$$
In particular, for $\Lambda$ large enough we must have $\sum_{b\in \Lambda^*} \E[1-\sigma_b^1\sigma_b^2]>c |\Lambda^*|$ for some $c>0$. This implies the claim.
Fix such an edge.

We consider the interpolation on this single edge $b$, that is, we take $J_\Lambda(t)$ as $J_e(t)=\ee^{-t}J_e+\sqrt{1-\ee^{-2t}}J_e'$ for $e=b$ and $J_e(t)=J_e$ for $e\neq b$.
In this setting, Lemma \ref{lem: variance} and the same reasoning as in the proof of Theorem \ref{thm: main} with $Q_\Lambda(\sigma^1,\sigma^2)$ replaced by $\sigma^1_b\sigma^2_b$ and $\|\sigma -\sigma'\|$ by $(2-2\sigma^2_b\sigma^2_b)^{1/2}$ gives the bound
$$
\begin{aligned}
&\var\Big(H_{\Lambda,\J}(\sigma^1(0))-H_{\Lambda,\J}(\sigma^2(0)) \Big)\\
&\geq 
 2\int_0^t\left\{\E\Big[1-\sigma^1_b(0)\sigma^2_b(0)\Big]-\sqrt{2}\sum_{i=1,2}\Big( \E\Big[1-\sigma^i_b(0)\sigma^i_b(s)\Big]\Big)^{1/2}\right\}\ee^{-s}\rd s\ .
\end{aligned}
$$
Proceeding as in \eqref{eqn: overlap estimate} (since the first term in the bracket is greater than $c$ independently of $\Lambda$) it remains to find for any $\e>0$ an event $A_\e$ and $C=C(\e)>0$ such that $\PP(A_\e)>1-\e$, and on $A_\e$
$$
\sigma_b^1(t)=\sigma_b^1(0) \qquad \forall t\leq C\ .
$$
By Proposition \ref{prop: criterion}, this holds if $F_b(t)>0$ for $ t\leq C$. 
We take $A_\e=B_\e \cap \widetilde B_\e$ for the events $B_\e=\{F_b(0)>\delta\}$ and $\widetilde B_\e=\{ (|J_b|\vee |J'_b|)<\widetilde C\}$.
Recall from Remark \ref{rem: examples} that $F_b(0)=2|J_b-\mathcal C_b|$ and that $J_b$ is independent of $\mathcal C_b$. 
In particular, we get as in \eqref{eqn: flex estimate} that $\PP(B_\e)>1-\e$ by picking $\delta$ small enough. 
Moreover, $\widetilde C$ can be taken large enough so that $\PP(\widetilde B_\e)>1-\e$. This implies $\PP(A_\e)>1-2\e$. 
On this event, we get the same way as in Proposition \ref{prop: flex diff} that
$$
F_b(t)\geq F_b(0)- 6\sqrt{t} \cdot (|J_b|\vee |J'_b|) \geq \delta - 6\widetilde C\sqrt{t}\ .
$$
It remains to take $C=(12\widetilde C\delta)^{-2}$ to ensure that $F_b(t)\geq \delta/2$ for $t\leq C$ thereby proving the corollary.
\end{proof}
One might expect the proof to hold by simply dropping all but a single edge in \eqref{eqn: variance}.
However, all couplings in $\Lambda$ would then be perturbed leading to a worse estimate of the flexibility in Proposition \ref{prop: flex diff}.

\section{Relations to Scaling Theories}
\label{sec:scaling}

Some of the above results have interesting consequences when combined with non-rigorous scaling theories of the spin glass phase that have been proposed in the theoretical physics literature~\cite{Mac84,BM85,FH86,FH88}.
The scaling-droplet picture is one of several competing theories attempting to describe the low-temperature thermodynamic properties of the spin glass phase, and the results presented elsewhere in this paper by themselves neither favor nor disfavor any of these. However, they do shed additional light on the consequences of some of the assumptions made in the scaling-droplet picture, and these will be discussed in this section. Because scaling theories represent a non-rigorous approach (so far) to the study of the spin glass phase, no attempt will be made at mathematical rigor in this section (although the conjectures and results will be stated precisely); our goal is simply to explore what our rigorous results imply for one approach to understanding finite-dimensional spin glasses.

Before turning to scaling theories, we present a simple bound on the parameter~$\alpha$ introduced in Definition~\ref{df: ADC} that provides a necessary condition for the presence of incongruence. This relies on an {\it upper\/} bound on fluctuations of (free) energy differences derived elsewhere but never published~\cite{AFunpub,NSunpub} (however, a statement and proof of the bound can be found in~\cite{Stein16}). The statement of the corresponding theorem is as follows:
\begin{thm}
\label{thm:upperbound} (Aizenman-Fisher-Newman-Stein)
 Let $F_P$ be the free energy of the finite-volume Gibbs state generated
  by Hamiltonian~(\ref{eqn: H}) in a box $\Lambda$ of volume
  $L^d$ using periodic boundary conditions, and let $F_{AP}$ be that generated
  using antiperiodic boundary conditions. Let $X_\Lambda = F_P -
  F_{AP}$. Then ${\rm Var}(X_\Lambda)\le{\rm const.}\times L^{d-1}$, where
  ${\rm Var}(\cdot)$ denotes the variance over all of the couplings inside
  the box.
\end{thm}

\begin{rem}
\rm
Although stated for periodic-antiperiodic boundary conditions, the theorem applies to any pair of gauge-related boundary conditions, such as two fixed~BC's.
\end{rem}

Theorem~\ref{thm:upperbound} has been proved only for finite volumes, but it is reasonable to expect that it applies equally well to free energy fluctuations in finite-volume restrictions of infinite-volume pure or ground states; i.e., for two pairs of boundary conditions arising from two putative ground or pure states drawn from the metastate. We therefore propose the following conjecture:

\begin{conj}
\label{conj:upperbound}
The variance bound of Theorem~\ref{thm:upperbound} extends to $\var\Big(H_{\Lambda,\J}(\sigma^1)-H_{\Lambda,\J}(\sigma^2)\Big)$ of Theorem~\ref{thm: main}; i.e., for any $\sigma^1$ and $\sigma^2$ chosen as in Theorem~\ref{thm: main},
\begin{equation}
\label{eq:upperbound}
\var\Big(H_{\Lambda,\J}(\sigma^1)-H_{\Lambda,\J}(\sigma^2)\Big)\le A L^{d-1}\, ,
\end{equation}
where $A>0$ is a constant and $|\Lambda|=L^d$.
\end{conj}
Because scaling relations are typically expressed in terms of $L$ rather than $|\Lambda|$, the relation $|\Lambda|=L^d$ will be assumed for the remainder of this section.

\medskip

Corollary~\ref{cor: ADC-> variance} and Conjecture~\ref{conj:upperbound} when combined lead immediately to our first result, which because it relies on a conjecture will be stated as a claim rather than as a corollary:

\begin{claim}
\label{claim:alphad}
In order for incongruent states to appear in the zero-temperature metastate, it is necessary that $\alpha\ge 1/d$.
\end{claim}

Consequently, incongruent states are ruled out if $\alpha<1/d$.

\subsection{Implications for Droplet-Scaling Theories}
\label{subsubsec:ds}

Droplet-scaling theories remain one of the main contenders for the correct description of the spin glass phase in finite dimensions.  The primary assumption~\cite{FH86,FH88} of the droplet picture of Ising spin glasses is well-known (in what follows, we restrict the discussion to zero temperature): in any dimension in which the spin glass phase exists, the minimal excitation above the ground state on length scale~$L$ about a fixed point (call it the origin) is a compact droplet of order~$L^d$ coherently flipped spins with an energy cost of $L^\theta$. The (dimension-dependent) exponent $\theta$ originally arose from scaling theories that examined the properties of a disordered zero-temperature fixed point; for any dimension in which a low-temperature spin glass phase exists, $\theta>0$. Fisher and Huse~(FH)  moreover argued that in any dimension, $\theta\le (d-1)/2$.

There are several possible versions of the droplet-scaling approach. In what follows, we will assume the simplest possible version --- what can justifiably be called a ``minimal'' droplet picture.

To begin, consider all compact, connected clusters of $N$ spins containing the origin and with $L^d\le N\le (2L)^d$. Then the droplet theory (at zero temperature) makes the following assumptions~\cite{FH86,FH88}:

\begin{enumerate}[label=(\roman*)]

\medskip

\item The distribution $\rho_L(E_L)$ of minimal droplet energies has the scaling form
\begin{equation}
\label{eq:scaling}
\rho_L(E_L)\approx\frac{1}{\Upsilon L^{\theta_{\rm e}}}{\tilde\rho}\left[\frac{E_L}{\Upsilon L^{\theta_{\rm e}}}\right]\nonumber\, ,
\end{equation}
where $\Upsilon$ is constant and of order the standard deviation of the coupling distribution and $\tilde\rho(0)>0$. In other words, the typical minimal droplet energy is order $L^{\theta_{\rm e}}$, but there is a probability falling off as $L^{-\theta_{\rm e}}$ that the minimal droplet energy is of  order~one. (The notation $\theta_{\rm e}$ --- ``${\rm e}$" for excitation --- rather than simply $\theta$ is ours and not FH's; the reason for this notation will be discussed momentarily.)

\medskip

\item The surface area of the droplet boundary scales as $L^{d_s}$, where $d-1<d_s<d$.~\footnote{At first glance it might appear that the condition $d_s < d$ is already incompatible with the existence
of incongruent states. However, it is neither a necessary nor sufficient condition for incongruence to be absent.}
(Recent simulations of a related quantity in \cite{WMK18}, namely the fractal dimension of the interface induced by switching from periodic to antiperiodic boundary conditions,  find that $d_s<d$ for $d<6$, and seems to approach $d$ at $d=6$.)

\medskip

\item Energy difference fluctuations are governed by a (dimension-dependent) ``stiffness exponent'' $\theta_{\rm s}$ (again, our notation, not FH's), which governs the size of the free energy fluctuations when one switches from, say, periodic to antiperiodic boundary conditions in a volume $\Lambda_L$. That is, using the notation of Theorem~\ref{thm:upperbound},
\begin{equation}
\label{eq:stiffness}
aL^{2\theta_{\rm s}}\le {\rm Var}(X_\Lambda)\le bL^{2\theta_{\rm s}}\nonumber\, ,
\end{equation}
where $0<a<b<\infty$ are constants. In order for a stable spin glass phase to exist in dimension~$d$, it is necessary that $\theta_{\rm s}>0$.

\medskip

\item Droplet excitation energies scale in the same way as ground state interface energies; i.e., $\theta_{\rm s}=\theta_{\rm e}$. 

\medskip

This last assumption leads to what we referred to earlier as a ``minimal'' droplet-scaling theory, and has been the subject of some debate (see, for example,~\cite{KP04,KB05,A90}). Additional exponents have been proposed in various places, and in non-scaling theories there are multiple types of excitations and interfaces with different exponents~\cite{KP04}. However, the version of droplet theory with $\theta_{\rm s}=\theta_{\rm e}$ is the simplest and cleanest, has been shown to hold in some special cases~\cite{BKM03}, and corresponds to the original theory as proposed by FH. We therefore adopt it in what follows, and hereafter set $\theta_{\rm e}=\theta_{\rm s}=\theta$.

\medskip

\item $\theta\le (d-1)/2$.

\medskip

At least insofar as this refers to the stiffness exponent, this has been rigorously proved, as noted above.

\end{enumerate}

\bigskip

After the scaling theories had been introduced, it was quickly noted that they implied what came to be known as ``temperature chaos'', namely a rearrangement on all sufficiently large lengthscales of the pure state correlations upon an infinitesimal change of temperature~\cite{BM87,FH88}. This is believed to be closely related to disorder chaos, and the behavior of the two is expected to be similar; in particular, the exponent governing the lengthscale beyond which edge and spin overlaps fall to zero is the same in all treatments of both kinds of spin glass ``chaos''. 

In analyzing disorder chaos, one begins by considering (see, for example,~\cite{KB05}) a small perturbation of the couplings of the form
\begin{equation}
\label{eq:pert}
J_{xy}\to J'_{xy}=\frac{J_{xy}+\eta_{xy}\Delta J}{\sqrt{1+(\Delta J)^2}}
\end{equation}
where $\eta_{xy}$ is a normally distributed random variable with zero mean and unit variance. 

Scaling theory then predicts~\cite{BM87,FH88,KB05} that a new ground state will appear outside of a characteristic length $\ell_c$ that is governed by the various exponents introduced above. We therefore add this to the above assumptions:

\medskip

(vi) Upon changing the couplings in the manner~(\ref{eq:pert}) above, the ground state rearranges beyond a lengthscale~$\ell_c$ governed by
\begin{equation}
\label{eq:lc}
\ell_c(\Delta J)=\Delta J^{-1/\xi}
\end{equation}
where the new exponent $\xi=d_s/2-\theta$. For a system of linear size $L$, the spin overlap $q$ obeys the scaling law~\cite{KB05}
\begin{equation}
\label{eq:spin}
\langle q(\Delta J,L)\rangle=F(L/\ell_c)=F(\Delta J^{1/\xi}L)
\end{equation}
where $F(x)\approx 1-ax^\xi$ for $x\ll 1$ and $F(x)\approx bx^{-d/2}$ for $x\gg 1$. The edge overlap behaves similarly. 

To express this using our notation, we relate $\Delta J$ in~(\ref{eq:pert}) and $t$ in~(\ref{eq:t}). For $\sqrt{t}<\epsilon\ll 1$ and $\Delta J<\epsilon\ll 1$, we have $\Delta J=(2t)^{1/2}+O(\epsilon^2)$.
Therefore to order $\epsilon^2$,
\begin{equation}
\label{eq:t1}
\ell_c= t^{-1/2\xi}\, .
\end{equation}
In~(\ref{eq:t1}) we rescaled $\ell_c$ by a factor of order one to eliminate a multiplicative constant. This has no effect on the analysis to follow.

\medskip

It should be emphasized that Eq.~(\ref{eq:lc}) of (vi) is not a separate assumption: it follows directly from (i) (at least for disorder chaos). For ease of presentation and future reference, it will be listed along with the assumptions above, but it should be kept in mind that it is a prediction of scaling theory, not an assumption.

\bigskip

We now turn to a discussion of what the results proved in this paper imply about the minimal scaling theory described above. One of the central conclusions of the droplet picture is that the ordered spin glass phase consists of a single pair of spin-reversed pure states (at $T>0$) or ground states (at $T=0$) in all dimensions in which a spin glass phase occurs~\cite{FH87}. The argument against the presence of incongruent states relied first on the inequality $\theta\le (d-1)/2$, which (at least as far as the spin glass stiffness is concerned) is no longer in dispute. The second part of the argument relied on a conjecture that the standard deviation of the energy fluctuations arising from the presence of incongruent states would scale at least as fast as the square root of the volume, leading to a contradiction. However, at the time this argument was put forward, there was little firm support for any sort of lower bound on (free) energy fluctuations arising from incongruence.~\footnote{However, as mentioned in the introduction, recent work by the authors in collaboration with J.~Wehr~\cite{ANSW14,ANSW16} has proved a lower bound for the variance scaling with the volume for at least certain pairs of incongruent states.} As a consequence, while it was generally accepted that the droplet theory leads to a two-state picture, the conclusion has remained mostly conjectural. (Some authors assert that the droplet picture simply assumes at the outset that there is only a single pair of ground states~\cite{KP04}.)

\medskip

However, we are now in a position to solidify the argument that the droplet theory indeed leads to a two-state picture, at least insofar as incongruence is concerned. The claim we make is the following: 
\begin{claim}
\label{claim:twostate}
If the scaling-droplet theory as defined by Assumptions (i)-(vi) above is correct, then in any finite dimension the zero-temperature metastate, generated using coupling-independent boundary conditions, is supported on a single pair of ground states.
\end{claim}

\begin{proof}[Proof of Claim~\ref{claim:twostate}]

By Definition~\ref{df: ADC}, the absence of disorder chaos on scale $\alpha$ means that
\begin{equation}
\label{eq:Q}
Q_\Lambda(\sigma^i, \sigma^i(t))>1-\e \text{ on $A_\e$, $i=1,2$,}
\end{equation} 
for all $t\leq C|\Lambda|^{-\alpha}$ and all $\Lambda$ large enough. Using Eq.~(\ref{eq:t1}), this leads to the identification
\begin{equation}
\label{eq:alphaxi}
\alpha=2\xi/d=d_s/d-2\theta/d\, .
\end{equation}
In a dimension with a spin glass phase $0<\theta<d_s/2$, so $0<\alpha<1$. 

Eq.~(\ref{eq:alphaxi}) says that the minimal droplet excitation in a volume of linear dimension~$L$ sets the scale for the absence of disorder chaos.
Moreover, Corollary~\ref{cor: ADC-> variance} says that, if there is ADC at scale $\alpha$, and ${\mathbb\sigma}^1$ and ${\mathbb\sigma}^2$ are incongruent spin configurations in $\Lambda$, then
\begin{equation}
\label{eq:corollary1}
{\rm Var}\left(H_{\Lambda,J}({\bf\sigma}^1)-H_{\Lambda,J}({\bf\sigma}^2)\right)\ge C|\Lambda|^{1-\alpha}\, .
\end{equation}
When combined with~(\ref{eq:alphaxi}), this becomes
\begin{equation}
\label{eq:corollary2}
{\rm Var}\left(H_{\Lambda,J}({\bf\sigma}^1)-H_{\Lambda,J}({\bf\sigma}^2)\right)\ge C L^{d-d_s+2\theta}\ge CL^{2\theta+\delta}\ ,
\end{equation}
where $\delta(d)\equiv d-d_s>0$ using assumption (ii) of the droplet theory.

But by Assumptions (iii) and (iv) of the droplet theory, 
\begin{equation}
\label{eq:corollary3}
{\rm Var}\left(H_{\Lambda,J}({\bf\sigma}^1)-H_{\Lambda,J}({\bf\sigma}^2)\right)\le bL^{2\theta}\, ,
\end{equation}
leading to a contradiction for sufficiently large $L$ and demonstrating the above claim that the minimal droplet theory is indeed a two-state theory.
\end{proof}

It is interesting to note that while the original two-state argument relied on the inequality $\theta\le (d-1)/2$, this is nowhere used in the above argument; indeed, at least for the purposes of the argument, $\theta$ can be anything at all. The other assumptions of the droplet theory were all necessary, however. Of particular interest is that the droplet {\it geometry\/} plays a crucial role in setting the scale of ground state energy difference fluctuations.

\subsection{Further Relations}
\label{subsubsec:further}

We conclude this discussion with an argument that uses no scaling assumptions and arrives at another relation connecting droplet geometries and energies to the presence or absence of incongruence. In this case, however, the droplets under consideration are not low-energy excitations above the ground state but rather the ``critical droplets'', introduced above Theorem~\ref{thm: ADC}, that measure the stability of a given ground state pair. Consider the critical droplet boundary $\partial\mathcal D_b$ (of energy order one). This naturally leads to a new exponent $d_f$, defined as $|\partial\mathcal D_b|={\rm const.}\ L^{d_f}$. Then Theorem~\ref{thm: ADC} provides the relation~$\alpha=2d_f/d$, and Corollary~\ref{cor: ADC-> variance} gives the bound
\begin{equation}
\label{eq:newbound}
{\rm Var}\left(H_{\Lambda,J}({\bf\sigma}^1)-H_{\Lambda,J}({\bf\sigma}^2)\right)\ge {\rm const.}\ L^{d(1-\alpha)}={\rm const.}\ L^{d-2d_f}\, .
\end{equation}
Combining this with~(\ref{eq:corollary3}) then implies that if $d_f<(d-2\theta)/2$, there cannot be incongruent ground states. This result bypasses the issue of whether ~$\theta_{\rm e}=\theta_{\rm s}$; only the ``stiffness''~$\theta_{\rm s}$ enters.

\bibliographystyle{plain}

\bibliography{bib_incongruent}

\begin{thebibliography}{10}

\bibitem{AT07}
R.~J. Adler and J.~E. Taylor.
\newblock {\em Random fields and geometry}.
\newblock Springer Monographs in Mathematics. Springer, New York, 2007.

\bibitem{AFunpub}
M.~Aizenman and D.~S. Fisher.
\newblock Unpublished.

\bibitem{AW90}
M.~Aizenman and J.~Wehr.
\newblock Rounding effects of quenched randomness on first-order phase
  transitions.
\newblock {\em Comm. Math. Phys.}, 130(3):489--528, 1990.

\bibitem{AD11}
L.-P. Arguin and M.~Damron.
\newblock Short-range spin glasses and random overlap structures.
\newblock {\em J. Stat. Phys.}, 143(2):226--250, 2011.

\bibitem{AD14}
L.-P. Arguin and M.~Damron.
\newblock On the number of ground states of the {E}dwards-{A}nderson spin glass
  model.
\newblock {\em Ann. Inst. H. Poincar\'e Probab. Statist.}, 50(1):28--62, 02
  2014.

\bibitem{ADNS10}
L.-P. Arguin, M.~Damron, C.~M. Newman, and D.~L. Stein.
\newblock Uniqueness of ground states for short-range spin glasses in the
  half-plane.
\newblock {\em Comm. Math. Phys.}, 300(3):641--657, 2010.

\bibitem{ANSW14}
L.-P. Arguin, C.~M. Newman, D.~L. Stein, and J.~Wehr.
\newblock Fluctuation bounds for interface free energies in spin glasses.
\newblock {\em J. Stat. Phys.}, 156(2):221--238, 2014.

\bibitem{ANSW16}
L.-P. Arguin, C.~M. Newman, D.~L. Stein, and J.~Wehr.
\newblock Zero-temperature fluctuations in short-range spin glasses.
\newblock {\em J. Stat. Phys.}, 163(5):1069--1078, 2016.

\bibitem{BKM03}
J.-P. Bouchaud, F.~Krzakala, and O.~C. Martin.
\newblock Energy exponents and corrections to scaling in {I}sing spin glasses.
\newblock {\em Phys. Rev. B}, 68:224404, 2003.

\bibitem{BM85}
A.~J. Bray and M.~A. Moore.
\newblock Critical behavior of the three-dimensional {I}sing spin glass.
\newblock {\em Phys. Rev. B}, 31:631--633, 1985.

\bibitem{BM87}
A.~J. Bray and M.~A. Moore.
\newblock Chaotic nature of the spin-glass phase.
\newblock {\em Phys. Rev. Lett.}, 58:57--60, 1987.

\bibitem{C14}
S.~Chatterjee.
\newblock {\em Superconcentration and related topics}.
\newblock Springer Monographs in Mathematics. Springer, Cham, 2014.

\bibitem{EA75}
S.~F. {Edwards} and P.~W. {Anderson}.
\newblock {Theory of spin glasses}.
\newblock {\em Journal of Physics F Metal Physics}, 5:965--974, May 1975.

\bibitem{FH86}
D.~S. Fisher and D.~A. Huse.
\newblock Ordered phase of short-range ising spin-glasses.
\newblock {\em Phys. Rev. Lett.}, 56:1601--1604, Apr 1986.

\bibitem{FH87}
D.~S. Fisher and D.~A. Huse.
\newblock Absence of many states in realistic spin glasses.
\newblock {\em J. Phys. A}, 20(15):L1005--L1010, 1987.

\bibitem{FH88}
D.~S. Fisher and D.~A. Huse.
\newblock Equilibrium behavior of the spin-glass ordered phase.
\newblock {\em Phys. Rev. B}, 38:386--411, Jul 1988.

\bibitem{HF87}
D.~A. Huse and D.~S. Fisher.
\newblock Pure states in spin glasses.
\newblock {\em Journal of Physics A: Mathematical and General}, 20(15):L997,
  1987.

\bibitem{KB05}
F.~Krzakala and J.-P. Bouchaud.
\newblock Disorder chaos in spin glasses.
\newblock {\em Europhys. Lett.}, 72:472--478, 2005.

\bibitem{KM00}
F.~Krzakala and O.~C. Martin.
\newblock Spin and link overlaps in three-dimensional spin glasses.
\newblock {\em Phys. Rev. Lett.}, 85:3013--3016, 2000.

\bibitem{KP04}
F.~Krzakala and G.~Parisi.
\newblock Local excitations in mean-field spin glasses.
\newblock {\em Europhys. Lett.}, 66:729--735, 2004.

\bibitem{Mac84}
W.~L. McMillan.
\newblock Scaling theory of {I}sing spin glasses.
\newblock {\em J. Phys. C}, 17:3179--3187, 1984.

\bibitem{MPSTV84a}
M.~M\'ezard, G.~Parisi, N.~Sourlas, G.~Toulouse, and M.~Virasoro.
\newblock Nature of spin-glass phase.
\newblock {\em Phys. Rev. Lett.}, 52:1156--1159, 1984.

\bibitem{MPSTV84b}
M.~M\'ezard, G.~Parisi, N.~Sourlas, G.~Toulouse, and M.~Virasoro.
\newblock Replica symmetry breaking and the nature of the spin-glass phase.
\newblock {\em J. Phys. (Paris)}, 45:843--854, 1984.

\bibitem{MPV87}
M.~M\'ezard, G.~Parisi, and M.~A. Virasoro.
\newblock {\em Spin Glass Theory and Beyond}.
\newblock World Scientific, Singapore, 1987.

\bibitem{NSunpub}
C.~M. Newman and D.~L. Stein.
\newblock Unpublished.

\bibitem{NS2000}
C.~M. Newman and D.~L. Stein.
\newblock Nature of ground state incongruence in two-dimensional spin glasses.
\newblock {\em Phys. Rev. Lett.}, 84:3966--3969, Apr 2000.

\bibitem{NS01}
C.~M. Newman and D.~L. Stein.
\newblock Interfaces and the question of regional congruence in spin glasses.
\newblock {\em Phys. Rev. Lett.}, 87:077201, 2001.

\bibitem{NS03}
C.~M. Newman and D.~L. Stein.
\newblock Topical {R}eview: {O}rdering and broken symmetry in short-ranged spin
  glasses.
\newblock {\em J. Phys.: Cond. Mat.}, 15:R1319 -- R1364, 2003.

\bibitem{PY00}
M.~Palassini and A.~P. Young.
\newblock Nature of the spin glass state.
\newblock {\em Phys. Rev. Lett.}, 85:3017--3020, 2000.

\bibitem{Parisi79}
G.~Parisi.
\newblock Infinite number of order parameters for spin-glasses.
\newblock {\em Phys. Rev. Lett.}, 43:1754--1756, 1979.

\bibitem{Parisi83}
G.~Parisi.
\newblock Order parameter for spin-glasses.
\newblock {\em Phys. Rev. Lett.}, 50:1946--1948, 1983.

\bibitem{SK75}
D.~Sherrington and S.~Kirkpatrick.
\newblock Solvable model of a spin glass.
\newblock {\em Phys. Rev. Lett.}, 35:1792--1796, 1975.

\bibitem{Stein16}
D.~L. Stein.
\newblock Frustration and fluctuations in systems with quenched disorder.
\newblock In P.~Chandra, P.~Coleman, G.~Kotliar, P.~Ong, D.L. Stein, and C.~Yu,
  editors, {\em {PWA90: A} Lifetime of Emergence}, pages 169--186. World
  Scientific, Singapore, 2016.

\bibitem{Subag16}
E.~Subag.
\newblock The geometry of the {G}ibbs measure of pure spherical spin glasses.
\newblock 2016.
\newblock Available at arXiv:1604.00679.

\bibitem{A90}
A.~van Enter.
\newblock Stiffness exponent, number of pure states, and {A}lmeida-{T}houless
  line in spin-glasses.
\newblock {\em J. Stat. Phys.}, 60:275--279, 1990.

\bibitem{WMK18}
Wang. W, M.~A. Moore, and H.~G. Katzgraber.
\newblock Fractal dimension of interfaces in {E}dwards-{A}nderson spin glasses
  for up to six space dimensions.
\newblock {\em Preprint, arxiv:1712.04971}, 2018.

\end{thebibliography}

\end{document}